\documentclass[11pt,reqno,UTF8,twoside]{amsart}
\usepackage{mathrsfs}
\usepackage{amsfonts,amssymb,amsmath,amsthm}
\usepackage{cite}
\usepackage[colorlinks=true,citecolor=red]{hyperref}
\usepackage{titletoc}
\usepackage{geometry}
\usepackage{bm}
\usepackage{indentfirst}
\usepackage{graphicx}
\usepackage{float} 
\usepackage{booktabs}
\usepackage{longtable}
\usepackage{subfigure}
\usepackage{stmaryrd}
\usepackage{enumerate}
\usepackage{tikz}
\usepackage{ulem}
\usepackage{color,soul}
\usepackage{setspace}
\usepackage[pagewise]{lineno} 
\allowdisplaybreaks[2]
\usepackage{appendix}
\usepackage{cases}
\usepackage{todonotes}
\usepackage{enumitem}

\numberwithin{equation}{section}
\newtheorem{theorem}{Theorem}[section]
\newtheorem{lemma}[theorem]{Lemma}
\newtheorem{corollary}[theorem]{Corollary}
\newtheorem{proposition}[theorem]{Proposition}

\newtheorem{remark}[theorem]{Remark}
\newtheorem{definition}[theorem]{Definition}

\newcommand{\N}{\mathbb{N}}
\newcommand{\R}{\mathbb{R}}

\newcommand{\E}{\mathbb{E}}

\newcommand{\dvg}{\operatorname{div}}

\newcommand{\runum}[1]{\romannumeral #1}
\hypersetup{linkcolor=blue}

\makeatletter
\@namedef{subjclassname@2020}{\textup{2020} Mathematics Subject
Classification}
\makeatother

\begin{document}
     
   \title[ Asymptotic strong Feller and weak observability inequality]
	{{\Large A{\MakeLowercase{symptotic} s\MakeLowercase{trong} F\MakeLowercase{eller and }w\MakeLowercase{eak }o\MakeLowercase{bservability} i\MakeLowercase{nequality}}
    }}

    \author[Z. Liu, S. Xiang]{ {\small Z\MakeLowercase{iyu} L\MakeLowercase{iu}, S\MakeLowercase{hengquan} X\MakeLowercase{iang}}}
    
    \address[Ziyu Liu]{School of Mathematics and Physics, University of Science and Technology Beijing, 100083, Beijing, China.}
    \email{ziyu@ustb.edu.cn}

    \address[Shengquan  Xiang]{School of Mathematical Sciences, Peking University, 100871, Beijing, China.}
    \email{shengquan.xiang@math.pku.edu.cn}
    
\begin{abstract}

For a class of non-autonomous linear SPDEs, we establish the equivalence among asymptotic regularization, weak observability, and approximate null controllability for the associated deterministic control systems. This equivalence provides a deterministic control-theoretic characterization of stochastic smoothing and offers a systematic approach to studying SPDEs driven by spatially localized noise.  We further establish a criterion for semilinear SPDEs based on weak observability of the linearized equations. Our approach combines methods from PDE control theory with Malliavin calculus. As applications, we consider the stochastic Oseen equation, non-autonomous uniformly parabolic equations, and the parabolic Sine--Gordon equation, all driven by finite-dimensional, spatially localized white-in-time noise.

\end{abstract}

	

  \maketitle
    \setcounter{tocdepth}{1}
    \tableofcontents

\section{Introduction} \label{Sec 1}

  \subsection{Background and motivation}\label{Sec 1.1}
   
The asymptotic strong Feller property, introduced by Hairer and Mattingly \cite{HM-06,HM-08}, is an  important notion of asymptotic regularity for Markov semigroups. It is a standard tool for the study of ergodic and mixing properties, and it weakens the classical strong Feller property;  see e.g. \cite{DPZ-96,PZ-95}. This weaker form is well suited to SPDEs with degenerate noise, where strong Feller is often unavailable.

\vspace{2mm} 

In the works \cite{HM-06,HM-08,HM-11b}, Hairer and Mattingly investigated ergodicity and mixing for the 2D stochastic Navier--Stokes equations, focusing on the highly degenerate setting in which the noise acts only on four Fourier modes, and introduced the asymptotic strong Feller property. In the study of mixing for Navier--Stokes equations driven by random forcing,  Shirikyan considered another direction based on coupling methods, where the noise is localized in physical space and bounded in time \cite{Shi-15,Shi-21}. Along this line, the authors of the present paper, together with their collaborators, recently obtained mixing results for wave and dispersive equations with spatially localized bounded noise \cite{LWXZZ-24,CXZZ-25}.

More recently, motivated by statistical properties of turbulent models under localized white-noise forcing, our work \cite{LXZ-26} took a first step in this direction by proving exponential mixing for the 1D stochastic Allen--Cahn equation with localized white noise.

\vspace{2mm}

For SPDEs with degenerate noise, proving the asymptotic strong Feller property is often delicate. A standard route, provided in \cite{HM-06,HM-08}, is to establish quantitative asymptotic regularization estimates for the Markov semigroup. In those works, this was achieved by combining Malliavin calculus with a detailed analysis of the Malliavin matrix.

\begin{definition}\label{def1}
Let $\{X_t\}_{t\geq0}$ be a Markov process on a Hilbert space $H$, and let $\{P_t\}_{t\geq0}$ be its Markov semigroup. Let $\alpha\in(0,1)$ and $T>0$. We say that $\{P_t\}_{t\geq0}$ is $\alpha$-asymptotically regular\footnote{This type of gradient estimate was introduced by Hairer and Mattingly \cite{HM-06,HM-08}, which captures the asymptotic regularizing property of the Markov semigroup. The label used here is only for convenience of exposition.} in time $T$ if there exists a constant $C>0$ such that, for any $x\in H$ and $\varphi\colon H\to\mathbb R$ with $\|\varphi\|_\infty$ and $\|\nabla\varphi\|_\infty$ finite, one has
\begin{equation*}
    |\nabla P_T\varphi(x)|\leq C\|\varphi\|_\infty+\alpha\|\nabla\varphi\|_\infty.
\end{equation*}
\end{definition}

Indeed, such asymptotic regularization  estimates can be linked, through Malliavin calculus, to PDE control theory; see e.g. \cite{LX-26,LXZ-26}. In this paper, we further prove that, for linear equations, this $\alpha$-asymptotic regularity is {\it equivalent} to the $\alpha$-weak observability and $\alpha$-null controllability (central problems in control theory; see Definitions~\ref{def2}, \ref{def3} below). 

\subsection{SPDE setting and PDE control}

For clarity, we first present the basic autonomous linear setting. The full non-autonomous version will be stated in Section~\ref{Sec 2}.

Let $(H,\langle\cdot,\cdot\rangle)$ and $U$ be two separable Hilbert spaces. We consider the abstract linear SPDE
\begin{equation}\label{abs equation}
\begin{cases}
     dX=AXdt+BdW_t,\quad t>0,\\
     X_0=x,
\end{cases}
\end{equation}
where $A:D(A)\subset H\to H$ generates a strongly continuous semigroup $\{S(t)\}_{t\geq0}$ on $H$, $W$ is a cylindrical Wiener process on $U$, and $B\in\mathcal L(U;H)$. Assume there exists $\beta\in(0,1/2)$ such that
\begin{equation*}
    \int_0^Tt^{-2\beta}\|S(t)B\|_{\mathcal L_2(U;H)}^2dt<\infty,
    \qquad \forall\,T>0,
\end{equation*}
where $\mathcal L_2(U;H)$ denotes the Hilbert--Schmidt operators from $U$ to $H$. Then by \cite{DPZ-96}, for any $x\in H$ and $T>0$, equation \eqref{abs equation} admits a unique mild solution
$X\in L^2(\Omega;C([0,T];H))$ satisfying
\begin{equation*}
    X_t=S(t)x+\int_0^t S(t-r)BdW_r,\qquad t\in[0,T].
\end{equation*} 

The gradient estimate in Definition~\ref{def1} is closely related to a deterministic control problem acting through the same directions as the noise. We therefore consider
\begin{equation}\label{control equation}
\begin{cases}
    dY=AYdt+Bh(t)dt,\quad t\in(0,T),\\
    Y_0=y,
\end{cases}
\end{equation}
where $h\in L^2(0,T;U)$ is the control. 

The associated notion  is the following $\alpha$-null controllability.   
\begin{definition}\label{def2}
Let $\alpha\in(0,1)$ and $T>0$. We say that equation \eqref{control equation} is $\alpha$-null controllable in time $T$ if there exists a constant $C>0$ such that, for any $y\in H$, there exists $h\in L^2(0,T;U)$ satisfying
\begin{equation*}
    \|Y_T\|\leq \alpha\|y\|,\qquad\|h\|_{L^2(0,T;U)}\leq C\|y\|.
\end{equation*}
\end{definition}

The duality between controllability and observability is well-known; see Coron \cite{Coron-07}. In this autonomous setting, the duality between weak observability and  $\alpha$-null controllability is established in \cite{LR-95, TWX-20}. This viewpoint  turns stabilization and approximate controllability questions into quantitative observability inequalities, which are often accessible by PDE methods. 

\begin{definition}\label{def3}
Let $\alpha\in(0,1)$ and $T>0$. We say that the $\alpha$-weak observability inequality holds in time $T$ if there exists a constant $C>0$ such that, for any $z\in H$,
\begin{equation*}
    \|S(T)^*z\|\leq C\|B^*S(T-\cdot)^*z\|_{L^2(0,T;U)}+\alpha\|z\|.
\end{equation*}
\end{definition}

In particular, observability inequalities of this form, with a remainder term, already appeared in Lebeau--Robbiano's work \cite{LR-95},
\begin{equation*}
    \partial_tu=\Delta u+\mathbf 1_{\omega}Q_Nf,
\end{equation*} 
where $Q_N$ denotes the projection onto the low-frequency subspace $E_N\subset L^2(D)$.  There, a spectral inequality
\begin{equation}\label{LR-low}
    \|z\|_{L^2(D)}\leq C_{N}\|\mathbf 1_{\omega}z\|_{L^2(D)}, \qquad z\in E_N,
\end{equation}
together with high-frequency dissipation quickly yields the weak observability
\begin{equation}\label{LR-low-weak}
    \|e^{T\Delta}z\|_{L^2(D)}\leq C_{\alpha}\|Q_N \mathbf 1_{\omega}e^{(T-\cdot)\Delta}z\|_{L^2(0,T;L^2(D))}+\alpha \|z\|_{L^2(D)},\qquad z\in L^2(D).
\end{equation}
The full observability inequality is then obtained by combining \eqref{LR-low-weak} with an infinite iteration in time; see e.g. \cite{Miller-10,Xiang-24} for further developments of this strategy from different perspectives. Weak observability inequalities also arise in many other contexts, including results in the whole space \cite{AM-22,WWZZ-19}, as well as discretization problems \cite{Zuazua04}.

Weak observability inequalities are also closely related to uncertainty principles, a nonzero function and its Fourier frequency  cannot be both sharply localized. In the heat equation example illustrated above, this is reflected in \eqref{LR-low}.\footnote{ We remark that there is another notion of weak observability for hyperbolic and other conservative systems. It often denotes observability with a loss of derivatives, formulated in a weaker energy topology.}

\subsection{Main results}
The basic linear relation between asymptotic regularization, weak observability and $\alpha$-null controllability is as follows, which is proved in Section~\ref{Sec 2} as the autonomous special case of Theorem~\ref{thm 2}.

\begin{theorem}\label{thm 1}
Let $T>0$. The following three statements are equivalent.
\begin{itemize}
    \item[$(\runum{1})$] {\rm($\alpha$-asymptotic regularity)} There exists $\alpha_1\in(0,1)$ such that the Markov semigroup $\{P_t\}_{t\geq 0}$ associated with equation \eqref{abs equation} is $\alpha_1$-asymptotically regular in time $T$.

    \vspace{1mm}

    \item[$(\runum{2})$] {\rm($\alpha$-weak observability)} There exists $\alpha_2\in(0,1)$ such that the $\alpha_2$-weak observability inequality holds in time $T$.

    \vspace{1mm}

    \item[$(\runum{3})$] {\rm($\alpha$-null controllability)} There exists $\alpha_3\in(0,1)$ such that equation \eqref{control equation} is $\alpha_3$-null controllable in time $T$.
\end{itemize}
Moreover, when $(\runum{2})$ or $(\runum{3})$ holds, the parameters $\alpha_1,\alpha_2,\alpha_3$ can be chosen to be the same.
\end{theorem}

\begin{remark}
{\rm
Theorem~\ref{thm 1} connects conditions from two different viewpoints. From the probabilistic side, asymptotic regularization is an important smoothing mechanism in the study of mixing and ergodicity for SPDEs driven by degenerate noise; see e.g. \cite{HM-06,FGRT-15,PZZ-24}. Indeed, for infinite-dimensional systems perturbed by degenerate or localized noise, the classical strong Feller property often fails, and one is naturally led to weaker asymptotic forms of regularization; see the discussion in \cite{GLLL-24}.

From the control viewpoint, this corresponds to weak observability.  When the control acts only through finitely many (localized) directions, the exact observability is usually out of reach and the natural estimate is of weak observability type \cite{LR-95}. This is in particular consistent with the probabilistic difficulty caused by  degenerate or localized noise, and is also linked to uncertainty principles,  given that only part of the system is directly forced. 

This connection provides a systematic approach to investigating the long-time dynamics of SPDEs driven by localized noise. In particular, it makes it possible to use tools developed in PDE control theory such as the moment method, Carleman estimates, and microlocal analysis to study such SPDE problems.

}
\end{remark}

\noindent {\bf Several generalizations.} The autonomous formulation above is chosen for clarity; the same approach extends in several directions, as explained below.

\begin{itemize}[leftmargin=2em]
    \item[1)] The same equivalence holds for {\it non-autonomous linear SPDEs}. This is proved in Theorem~\ref{thm 2}, where the Markov semigroup is replaced by a two-parameter Markov family. Such non-autonomous Markov families appear naturally in time-dependent SPDEs.

    \vspace{2mm}
    
    \item[2)] We also treat {\it semilinear equations with Lipschitz nonlinearities}. Theorem~\ref{thm 3} shows that asymptotic regularization follows from weak observability  for the  linearized equation. This can be viewed as an asymptotic counterpart of the classical strong Feller criterion for semilinear equations under controllability assumptions; see e.g. \cite[Theorem 7.2.4]{DPZ-96}.

    \vspace{2mm}

    \item[3)] The control/noise operator $B$ may also be {\it unbounded}. This is natural from the control viewpoint, as unbounded operators commonly arise from boundary controls; see e.g. \cite{MRT-26}. On the stochastic side, unbounded noise coefficients also appear in strong Feller regularization results; see e.g. \cite{LX-26,GP-26}.

    \vspace{2mm}

    \item[4)] In particular, this approach could be adapted to other {\it nonlinear SPDEs} driven by localized white noise. For general nonlinearities, verifying the required weak observability and Malliavin estimates usually depends on the specific equation and requires case-by-case analysis; see e.g. \cite{LXZ-26}.
\end{itemize}

\subsection{Applications} We now turn to three SPDE systems that are driven by physically localized and finite-dimensional white-in-time noise, illustrating respectively the autonomous linear, non-autonomous linear, and semilinear criteria.

\vspace{2mm}

\noindent {\it The stochastic Oseen equation.}
The first example is the stochastic Oseen equation,
\begin{equation*}
    \partial_tu-\nu\Delta u+(\overline U\cdot\nabla)u+(u\cdot\nabla)\overline U+\nabla p=\tfrac{dW}{dt},\qquad\dvg u=0,
\end{equation*}
where $\overline U$ is a prescribed stationary background flow. This equation is the linearization of the Navier--Stokes equations around $\overline U$, and its controllability is closely related to the controllability analysis of the Navier--Stokes system; see e.g. \cite{BT-25}.  Using Theorem \ref{thm 1}, we prove that it is $\alpha$-asymptotically regular for any $\alpha\in(0,1)$, once sufficiently many localized noise modes are forced. See Section \ref{Sec Oseen} for details.

 As the Oseen dynamics need not be exponentially stable, the asymptotic regularization cannot rely only on deterministic decay. The control argument compensates for possible unstable directions via weak observability: by forcing sufficiently many localized noise modes, the remainder can be made smaller than any prescribed $\alpha\in(0,1)$, yielding $\alpha$-asymptotic regularization.

\vspace{2mm}

\noindent {\it Non-autonomous uniformly parabolic equations.}
The second example concerns non-autonomous uniformly parabolic equations of the form
\begin{equation*}
    \partial_tu-\dvg(a(t,x)\nabla u)+b(t,x)\cdot\nabla u+c(t,x)u=\tfrac{dW}{dt},
\end{equation*}
with uniformly elliptic and sufficiently regular coefficients. Controllability of parabolic equations has been extensively studied, in particular through Carleman estimates; see e.g. \cite{FI-96,FCZ-00,FLZ-19}. Based on Theorem~\ref{thm 2}, we obtain its $\alpha$-asymptotic regularity. See Section \ref{Sec 5} for details.   

Non-autonomous SPDEs have been studied in recent years because their long-time dynamics have many phenomena, involving periodic, quasi-periodic, or more general time-dependent invariant structures. This example shows that the weak-observability mechanism is compatible with such time-dependent dynamics.

\vspace{2mm}

\noindent {\it The parabolic Sine--Gordon equation.}
The last example is the parabolic Sine--Gordon equation 
\begin{equation*}
    \partial_tu-\Delta u+\kappa\sin u=\tfrac{dW}{dt}.
\end{equation*}
It is a parabolic analogue of the classical Sine--Gordon model, which gives a semilinear equation with Lipschitz nonlinearity. By applying  the semilinear criteria in  Theorem~\ref{thm 3}, we prove its $\alpha$-asymptotic regularity. See Section \ref{Sec 6} for details.

For SPDEs with Lipschitz nonlinearities and non-degenerate noise, Da Prato and Zabczyk \cite{DPZ-96} provided classical conditions for ergodicity. Meanwhile, the present work deals with localized degenerate noise; the asymptotic regularization obtained above can then be combined with the asymptotic strong Feller criterion of \cite{HM-08} to derive exponential mixing in Wasserstein distances.

\subsection*{Organization of the paper}

This paper is organized as follows. Section~\ref{Sec 2} proves the equivalence between asymptotic regularization, weak observability and approximate null controllability for non-autonomous linear equations. Section~\ref{Sec 3} establishes an abstract criterion for semilinear equations. The remaining sections apply these criteria to concrete models: the stochastic Oseen equation in Section~\ref{Sec Oseen}, non-autonomous uniformly parabolic equations in Section~\ref{Sec 5}, and the parabolic Sine--Gordon equation in Section~\ref{Sec 6}.

    \section{Equivalence for linear equations} \label{Sec 2}

    In this section, we investigate the relation between the asymptotic strong Feller property and weak observability inequalities for non-autonomous linear systems, leading to Theorem~\ref{thm 2}. As a special case, the autonomous result stated in Theorem~\ref{thm 1} follows directly.

\vspace{2mm}
 Recall that $H$ and $U$ are two separable Hilbert spaces. We consider the non-autonomous linear SPDE on $H$
\begin{equation}\label{abs equation2}
\begin{cases}
     dX(t)=A(t)X(t)dt+G(t)dt+B(t)dW_t,\quad t>0,\\
     X(0)=x,
\end{cases}
\end{equation}
where $G\in L^2_{\mathrm{loc}}(\mathbb R^+;H)$ and $W$ is a cylindrical Wiener process on $U$. The family of linear operators
\begin{equation*}
    A(t):D(A(t))\subset H\to H,\qquad t\geq 0,
\end{equation*}
generates a strongly continuous evolution family $\{U(t,s):0\leq s\leq t<\infty\}$ on $H$. Namely,
\begin{equation*}
    U(t,t)=I,\qquad U(t,r)U(r,s)=U(t,s),\qquad 0\leq s\leq r\leq t,
\end{equation*}
and $(t,s)\mapsto U(t,s)x$ is continuous for every $x\in H$. We also assume that, for every $T>0$, there exist constants $M_T>0$ and $c_T\in\mathbb{R}$ such that
\begin{equation*}
\|U(t,s)\|_{\mathcal L(H)}\leq M_T \exp(c_T(t-s)),\qquad 0\leq s\leq t\leq T.
\end{equation*}
Assume that $B\in L^2_{\rm loc}(\R^+;\mathcal{L}(U;H))$ and that there exists a constant $\beta\in(0,1/2)$ such that 
\begin{equation*}
    \sup_{0\leq t\leq T}\int_0^t(t-s)^{-2\beta}\|U(t,s)B(s)\|_{\mathcal L_2(U;H)}^2ds<\infty\quad\forall\, T>0.
\end{equation*}

Under these settings, following \cite[Section 5]{DPZ-96} and performing a standard  factorization argument, the stochastic convolution $\int_0^t U(t,s)B(s)dW_s$ is well defined and admits an $H$-valued continuous version. Hence, for every $x\in H$ and $T>0$, the process defined by
\begin{equation*}
    X(t)=U(t,0)x+\int_0^tU(t,s)G(s)ds+\int_0^tU(t,s)B(s)dW_s,\qquad 0\leq t\leq T,
\end{equation*}
belongs to $L^2(\Omega;C([0,T];H))$ and is the unique mild solution of \eqref{abs equation2}.

\vspace{2mm}

To introduce the Markov family
associated with \eqref{abs equation2}, for $0\leq s\leq t$ and
$x\in H$, let $X^{s,x}(t)$ denote the mild solution starting from $x$
at time $s$, i.e.,
\begin{equation*}
    X^{s,x}(t)=U(t,s)x+\int_s^tU(t,r)G(r)dr+\int_s^tU(t,r)B(r)dW_r,\qquad t\geq s.
\end{equation*}
Then for every bounded Borel measurable function $\varphi\colon H\to\mathbb R$, the two-parameter Markov family $\{P_{s,t}\}_{t\geq s\geq 0}$  is defined as
\begin{equation*}
    P_{s,t}\varphi(x):=\mathbb E\varphi(X^{s,x}(t)),\qquad 0\leq s\leq t,\;x\in H.
\end{equation*} 

We also consider the following controlled linear equation
\begin{equation}\label{eq Y-non}
\begin{cases}
    dY(t)=A(t)Y(t)dt+B(t)h(t)dt,\quad t\in(s,T),\\
    Y(s)=y.
\end{cases}
\end{equation}

    \begin{theorem}\label{thm 2}
         Let $T>s\geq 0$. The following three statements are equivalent.
        \begin{itemize}
             \item[$(\runum{1})$] There exist constants $\alpha_1\in(0,1)$ and $C_1>0$ such that the Markov family $\{P_{s,t}\}_{t\geq s\geq 0}$ associated with equation \eqref{abs equation2} is $\alpha_1$-asymptotically regular from time $s$ to time $T$ in the following sense: for any $x\in H$ and  $\varphi\colon H\rightarrow \R$  with $\|\varphi\|_{\infty}$ and $\|\nabla\varphi\|_{\infty}$ finite,
        \begin{equation*}
        |\nabla P_{s,T}\varphi(x)|\leq C_1\|\varphi\|_{\infty}+\alpha_1\|\nabla\varphi\|_{\infty}.
        \end{equation*}  
            \item[$(\runum{2})$] There exist constants $\alpha_2\in(0,1)$ and $C_2>0$ such that, for any $z\in H$,
            \begin{equation*}
            \|U(T,s)^*z\|\leq C_2\|B(\cdot)^*U(T,\cdot)^*z\|_{L^2(s,T;U)}+\alpha_2\|z\|.
        \end{equation*}
          \item[$(\runum{3})$] There exist constants $\alpha_3\in(0,1)$ and $C_3>0$ such that, for any $y\in H$, there exists $h\in L^2(s,T;U)$ such that the solution $Y$ of equation \eqref{eq Y-non} satisfies
            \begin{equation*}
                \|Y(T)\|\leq\alpha_3\|y\|,\quad \|h\|_{L^2(s,T;U)}\leq C_3\|y\|.
            \end{equation*}
        \end{itemize}        
          Moreover, $\alpha_2,\alpha_3$ in $(\runum{2}),(\runum{3})$ can be chosen to be the same. Additionally, when $(\runum{2})$ or $(\runum{3})$ holds, $\alpha_1$ in $(\runum{1})$  can be chosen as $\alpha_2$ or $\alpha_3$, respectively.       
    \end{theorem}

\begin{remark} The proof below shows that, the assumption  $(\runum{1})$ can be indeed replaced by a weaker version, namely, its pointwise version at the origin, 
\begin{itemize}
             \item[$(\runum{1}')$] There exist constants $\alpha_1\in(0,1)$ and $C_1>0$ such that the Markov family $\{P_{s,t}\}_{t\geq s\geq 0}$ associated with equation \eqref{abs equation2} satisfies that for $\varphi\colon H\rightarrow \R$  with $\|\varphi\|_{\infty}$ and $\|\nabla\varphi\|_{\infty}$ finite,
              \begin{equation*}
        |\nabla P_{s,T}\varphi(0)|\leq C_1\|\varphi\|_{\infty}+\alpha_1\|\nabla\varphi\|_{\infty}.
        \end{equation*}  
        \end{itemize}

\end{remark}

\begin{proof}
The proof consists of four parts. In the first two steps, we quickly prove the equivalence between the weak observability inequality and approximate null controllability for the linear control system, which is standard. Next, we derive the weak observability inequality from the asymptotic regularity estimate by testing it on oscillatory Gaussian functions. Finally, we use approximate null controllability and Malliavin integration by parts to obtain the gradient estimate for the Markov family.

To begin with, define
\begin{align*}
    &L_{s,T}\colon L^2(s,T;U)\rightarrow H,\quad  L_{s,T}h:=\int_s^T U(T,r)B(r)h(r)dr,\\
    &L_{s,T}^*\colon H\rightarrow L^2(s,T;U),\quad (L_{s,T}^*z)(r)=B(r)^*U(T,r)^*z,\quad r\in(s,T).
\end{align*}
For simplicity, write $L=L_{s,T}$ and $L^*=L_{s,T}^*$.

\vspace{2mm}

\noindent $(\runum{3})\Rightarrow(\runum{2}).$ Since
\begin{equation*}
    Y(T)=U(T,s)y+Lh,
\end{equation*}
statement $(\runum{3})$ is equivalent to the following assertion: for any $y\in H$, one can find $h\in L^2(s,T;U)$ such that
\begin{equation*}
    \|Lh+U(T,s)y\|\leq\alpha_3\|y\|,\quad \|h\|_{L^2(s,T;U)}\leq C_3\|y\|.
\end{equation*}
Therefore we derive that for any $y,z\in H$,
\begin{align*}
    |\langle U(T,s)^*z,y\rangle|&=|\langle z,U(T,s)y+Lh\rangle-\langle z,Lh\rangle|\leq|\langle z,U(T,s)y+Lh\rangle|+|\langle L^*z,h\rangle|\\
    &\leq \alpha_3\|z\|\|y\|+C_3\|L^*z\|\|y\|.
\end{align*}
Thus by taking the supremum over all $\|y\|\leq 1$, we derive that 
\begin{equation*}
    \|U(T,s)^*z\|\leq C_3\|L^*z\|_{L^2(s,T;U)}+\alpha_3\|z\|,
    \quad \forall\,z\in H,
\end{equation*} 
which implies $(\runum{2})$ with $\alpha_2=\alpha_3$, $C_2=C_3$.

\vspace{2mm}

\noindent $(\runum{2})\Rightarrow(\runum{3}).$ For any fixed $y\in H$, we define the convex set $K_y$ by 
\begin{equation*}  
K_y:=L\left(\overline{B}_{L^2(s,T;U)}\left(0,C_2\|y\|\right)\right)+\overline{B}_H\left(0,\alpha_2\|y\|\right).
\end{equation*}
Since $L^2(s,T;U)$ and $H$ are Hilbert spaces, the two closed balls above are weakly compact. As $L$ is bounded and linear, the image of the first ball under $L$ is weakly compact. Hence $K_y$ is weakly compact and therefore norm closed. The support function of $K_y$ is defined by
\begin{equation*}
    \sigma_y(z):=\sup_{k\in K_y}\langle z,k\rangle=C_2\|L^*z\|_{L^2(s,T;U)}\|y\|+\alpha_2\|z\|\|y\|,\quad z\in H.
\end{equation*}
Then, for any $y,z\in H$, one has
\begin{equation*}
    \langle -U(T,s)y,z\rangle\leq \|y\|\|U(T,s)^*z\|\leq C_2\|L^*z\|_{L^2(s,T;U)}\|y\|+\alpha_2\|z\|\|y\|=\sigma_y(z).
\end{equation*}
If $-U(T,s)y\notin K_y$, then by the Hahn–Banach separation theorem, there exists $z_0\in H$ such that 
\begin{equation*}
   \langle-U(T,s)y,z_0\rangle>\sup_{k\in K_y}\langle z_0,k\rangle=\sigma_y(z_0), 
\end{equation*}
which contradicts the preceding inequality. Thus $-U(T,s)y\in K_y$.  Hence there exist $h\in L^2(s,T;U)$ and $w\in H$ such that
\begin{equation*}
    \|w\|\leq \alpha_2\|y\|,\quad \|h\|_{L^2(s,T;U)}\leq C_2\|y\|,\quad -U(T,s)y=Lh+w.
\end{equation*}
This proves $(\runum{3})$ with $\alpha_3=\alpha_2$ and $C_3=C_2$.
 
\vspace{2mm}

\noindent $(\runum{1})\Rightarrow(\runum{2}).$ By assumption, for any $x\in H$ and $\varphi\colon H\rightarrow\R$ with $\|\varphi\|_{\infty}$ and $\|\nabla\varphi\|_{\infty}$ finite,
\begin{equation}\label{eq alpha1 nonauto}
    |\nabla P_{s,T}\varphi(x)|\leq C_1\|\varphi\|_{\infty}+\alpha_1\|\nabla\varphi\|_{\infty}.
\end{equation}
For any $r>0$ and $z\in H$, let
\begin{equation*}
    \varphi_{r,z}(x):=\sin\left(r\left\langle z,x-\int_s^T U(T,l)G(l)dl\right\rangle\right),
    \quad x\in H.
\end{equation*}
Then $\|\varphi_{r,z}\|_{\infty}=1$ and $\|\nabla\varphi_{r,z}\|_{\infty}\leq r\|z\|$. For the solution starting from $x$ at time $s$, we have
\begin{equation*}
    X^{s,x}(T)=U(T,s)x+\int_s^TU(T,r)G(r)dr+\int_s^TU(T,r)B(r)dW_r.
\end{equation*}
Thus $X^{s,x}(T)$ is a Gaussian random variable with mean $U(T,s)x+\int_s^T U(T,r)G(r)dr$ and covariance operator $LL^*$. Therefore, a direct computation gives
\begin{equation*}
    P_{s,T}\varphi_{r,z}(x)=\exp\left(-\tfrac{r^2}{2}\|L^*z\|_{L^2(s,T;U)}^2\right)\sin\left(r\langle z,U(T,s)x\rangle\right).
\end{equation*}
Consequently, for any $y\in H$,
\begin{equation*}
    \langle\nabla P_{s,T}\varphi_{r,z}(0),y\rangle=r\exp\left(-\tfrac{r^2}{2}\|L^*z\|_{L^2(s,T;U)}^2\right)\langle z,U(T,s)y\rangle.
\end{equation*}
If $\|L^*z\|_{L^2(s,T;U)}=0$, then \eqref{eq alpha1 nonauto} implies that
\begin{equation*}
    r|\langle z,U(T,s)y\rangle|\leq C_1+\alpha_1r\|z\|,\quad\forall\,y\in H,\;\|y\|\leq1,\;r>0.
\end{equation*}
Letting $r$ tend to $\infty$, one has
\begin{equation*}
    \|U(T,s)^*z\|\leq\alpha_1\|z\|.
\end{equation*}
If $\|L^*z\|_{L^2(s,T;U)}>0$, then taking $r=k\|L^*z\|_{L^2(s,T;U)}^{-1}$ with $k>0$ in \eqref{eq alpha1 nonauto}, we derive that
\begin{equation*}
    \|U(T,s)^*z\|\leq C_1k^{-1}e^{\frac{k^2}{2}}\|L^*z\|_{L^2(s,T;U)}+\alpha_1 e^{\frac{k^2}{2}}\|z\|.
\end{equation*}

Choosing $k>0$ sufficiently small so that $\alpha_1 e^{\frac{k^2}{2}}:=\alpha_2\in(0,1)$ then yields statement $(\runum{2})$.

\vspace{2mm}

\noindent $(\runum{3})\Rightarrow(\runum{1}).$ Similar to \cite{LXZ-26}, the proof invokes
the Malliavin calculus. To indicate the initial time, initial condition,
and the random input, we denote the solution by $X^{s,x}(t;W)$. For
$v\in L^2(s,T;U)$, define the Malliavin derivative along $v$ by
\begin{equation*}
    \langle \mathcal{D}X^{s,x}(t;W),v\rangle_{L^2(s,T;U)}=\lim\limits_{\varepsilon\rightarrow0}\frac{X^{s,x}(t;W+\varepsilon\int_s^\cdot v(r)dr)-X^{s,x}(t;W)}
    {\varepsilon},
    \quad t\in[s,T],
\end{equation*}
where the limit is taken in $L^2(\Omega;H)$.\footnote{We also use below the Malliavin chain rule: for $\varphi$ with  $\|\varphi\|_{\infty}$ and $\|\nabla\varphi\|_{\infty}$ finite, $\varphi(X^{s,x}(T))$ is Malliavin differentiable, and its Malliavin derivative is obtained by composing $\mathcal{D}X^{s,x}(T)$ with $\nabla\varphi(X^{s,x}(T))$; see, e.g., \cite{Nualart-06}.}  For $V(t):=\langle \mathcal{D}X^{s,x}(t;W),v\rangle_{L^2(s,T;U)}$, it satisfies
\begin{equation*}
\begin{cases}
dV(t)=A(t)V(t)dt+B(t)v(t)dt,\quad t\in(s,T),\\ 
V(s)=0.
\end{cases}
\end{equation*}
On the other hand, the linearization $Z$ of  equation \eqref{abs equation2}  satisfies\footnote{Here $Z$ is deterministic, since it is the linearization of a linear SPDE. In the nonlinear case, the corresponding formulation contains stochastic terms and is more involved; see, e.g., \cite{LXZ-26}.} 
\begin{equation*}
\begin{cases}
dZ(t)=A(t)Z(t)dt,\quad t\in(s,T),\\
Z(s)=y.
\end{cases}
\end{equation*}
In particular,
\begin{equation*}
Z(T)=U(T,s)y.
\end{equation*}

Set $R(t):=Z(t)-V(t)$, which satisfies
\begin{equation*}
\begin{cases}
dR(t)=A(t)R(t)dt-B(t)v(t)dt,\quad t\in(s,T),\\
R(s)=y.
\end{cases}
\end{equation*}

Applying $(\runum{3})$ to the initial datum $y\in H$ and then replacing $h$ by $-v$, we can choose $v\in L^2(s,T;U)$ such that
\begin{equation*}
\|R(T)\|\leq\alpha_3\|y\|,\quad\|v\|_{L^2(s,T;U)}\leq C_3\|y\|.
\end{equation*}

We now compute that
\begin{equation*}
\begin{aligned}
    |\langle \nabla P_{s,T}\varphi(x),y\rangle|
    &=\left|\E\langle\nabla\varphi(X^{s,x}(T)),Z(T)\rangle\right|\\
    &\leq\left|\E\langle\nabla\varphi(X^{s,x}(T)),V(T)\rangle\right|+\left|\E\langle\nabla\varphi(X^{s,x}(T)),R(T)\rangle\right|\\
    &=\left|\E\langle \mathcal{D}(\varphi(X^{s,x}(T))),v\rangle_{L^2(s,T;U)}\right|+\left|\E\langle\nabla\varphi(X^{s,x}(T)),R(T)\rangle\right|\\
    &=\left|\E\left(\varphi(X^{s,x}(T))\int_s^T v(t)dW_t\right)\right|+\left|\E\langle\nabla\varphi(X^{s,x}(T)),R(T)\rangle\right|\\
    &\leq\|\varphi\|_{\infty}\E\left|\int_s^T v(t)dW_t\right|+\|\nabla\varphi\|_{\infty}\E\|R(T)\|\\
    &\leq C_3\|\varphi\|_{\infty}\|y\|+\alpha_3\|\nabla\varphi\|_{\infty}\|y\|.
\end{aligned}
\end{equation*}
Here the fourth line follows from the integration by parts formula in Malliavin calculus, see e.g. \cite{Nualart-06}. Since $v\in L^2(s,T;U)$ is deterministic, $\int_s^T v(t)dW_t$ is the usual Itô integral. Taking the supremum over all $y\in H$ with $\|y\|\leq1$, we obtain
\begin{equation*}
    |\nabla P_{s,T}\varphi(x)|\leq C_3\|\varphi\|_{\infty}+\alpha_3\|\nabla\varphi\|_{\infty}.
\end{equation*}
This gives the desired asymptotic regularity estimate with $\alpha_1=\alpha_3$. 

The proof of Theorem~\ref{thm 2} is complete.
\end{proof}

    \section{Asymptotic strong Feller for nonlinear equations} \label{Sec 3}

    In this section, we establish a criterion for the asymptotic strong Feller property for nonlinear SPDEs in terms of a weak observability inequality; see Theorem~\ref{thm 3}.

    \vspace{2mm}

    Recall that $H$ and $U$ are two separable Hilbert spaces. Consider the following equation
    \begin{equation}\label{nl equation}
    \begin{cases}
        dX=AXdt+F(X)dt+BdW_t,\quad t>0,\\
        X_0=x.
    \end{cases}
    \end{equation}
    Here $A:D(A)\subset H\to H$ generates an analytic semigroup $\{S(t)\}_{t\geq 0}$ on $H$, and $W$ is a cylindrical Wiener process on  $U$.  We also assume that there exists a Hilbert space $V$ continuously embedded in $H$ and a constant $\gamma\in[0,1/2)$ such that the semigroup satisfies the smoothing estimate
\begin{equation*}
    \|S(t)\|_{\mathcal L(H;V)}\leq C_Tt^{-\gamma},\qquad 0<t\leq T.
\end{equation*} 
We further assume that $B\in\mathcal{L}_2(U;H)$, and that $F\colon H\rightarrow H$ satisfies
    \begin{align*} 
    \sup_{x\in H}\frac{\|F(x)\|}{1+\|x\|}<\infty, \quad\sup_{x\in H}\|\nabla F(x)\|_{\mathcal{L}(H)}<\infty,  
    \end{align*}
    where $\nabla F(x)\in\mathcal L(H)$ denotes the Gâteaux derivative of $F$ at $x$, and $\nabla F$ is continuous with respect to the strong operator topology. Moreover, $F$ is twice differentiable along directions in $V$, and there exists $C>0$ such that, for all $x\in H$ and $h_1,h_2\in V$,
    \begin{equation}\label{eq F second V}
    \|\nabla^2F(x)(h_1,h_2)\|_H\leq C\|h_1\|_V\|h_2\|_V .
    \end{equation}

Under these settings,  for any $x\in H$ and $T>0$, equation \eqref{nl equation} admits a unique mild solution $X\in L^2(\Omega;C([0,T];H))$; see e.g. \cite{DPZ-96}.

    \vspace{2mm}
    To address the relation between asymptotic strong Feller and weak observability, let us consider a non-autonomous  controlled equation
    \begin{equation}\label{control equation2}
    \begin{cases}
        dY=AYdt+\nabla F(\xi)Ydt+Bh(t)dt,\quad t\in(0,T),\\
        Y_0=y,
    \end{cases}
    \end{equation}
    where $\xi\in C([0,T];H)$ is a potential and  $h\in L^2(0,T;U)$ denotes a control term. The mild solution of equation \eqref{control equation2} is given by
    \begin{equation*}
        Y_t=\mathcal{U}_{\xi}(t,0)y+\int_0^t\mathcal{U}_{\xi}(t,s)Bh(s)ds,\quad 0\leq t\leq T.
    \end{equation*}
    Here $\{\mathcal{U}_{\xi}(t,s)\}_{0\leq s\leq t\leq T}\subset \mathcal{L}(H)$ is the two-parameter family satisfying
    \begin{equation}\label{eq UZ}
        \mathcal{U}_{\xi}(t,s)=S(t-s)+\int_s^tS(t-r)\nabla F(\xi_r)\mathcal{U}_{\xi}(r,s)dr,\quad 0\leq s\leq t\leq T.
    \end{equation}

\begin{theorem}\label{thm 3}
    Let $T>0$. Assume that there exist constants $\alpha\in(0,1)$ and $C>0$ such that for any $\xi\in C([0,T];H)$, the $\alpha$-weak observability inequality holds in time $T$:
    \begin{equation}\label{eq L obs}
        \|\mathcal{U}_{\xi}(T,0)^*z\|\leq C\|B^*\mathcal{U}_{\xi}(T,\cdot)^*z\|_{L^2(0,T;U)}+\alpha\|z\|\quad \forall\, z\in H.
    \end{equation}
    Then the Markov semigroup $\{P_t\}_{t\geq0}$ associated with equation \eqref{nl equation} is $\alpha'$-asymptotically regular in time $T$ for any $\alpha'\in(\alpha,1)$.
\end{theorem}

The weak observability condition \eqref{eq L obs} can be formulated equivalently in terms of approximate null controllability. Indeed, for each fixed $\xi\in C([0,T];H)$, equation \eqref{control equation2} is a non-autonomous linear control system whose evolution family is $\mathcal U_\xi(t,s)$. Thus, applying the equivalence between weak observability and $\alpha$-null controllability proved in Theorem~\ref{thm 2}, we obtain the following.
\begin{corollary}\label{cor nl control}
    Let $T>0$. Assume that there exist constants $\alpha\in(0,1)$ and $C>0$ such that for any $\xi\in C([0,T];H)$ and $y\in H$, there exists $h\in L^2(0,T;U)$ such that the solution $Y$ of equation \eqref{control equation2} satisfies
    \begin{equation*}
        \|Y(T)\|\leq\alpha\|y\|,\quad\|h\|_{L^2(0,T;U)}\leq C\|y\|.
    \end{equation*}
    Then the Markov semigroup $\{P_t\}_{t\geq0}$ associated with equation \eqref{nl equation} is $\alpha'$-asymptotically regular in time $T$ for any $\alpha'\in(\alpha,1)$.
\end{corollary}

\begin{proof} [Proof of Theorem \ref{thm 3}]
The proof is divided into four steps. Step 1 constructs, by duality and weak observability, a control $h_\xi$ with bounded cost and terminal defect below $\alpha'\|y\|$. Step 2 applies this control to the stochastic path $X$ and decomposes the derivative flow into a Malliavin part and a remainder. Step 3 estimates the control cost through the generalized Itô isometry. Step 4 proves the required bounds on $\mathcal D_s\mathcal A_X$ and $\mathcal D_s\mathcal U_X$, which closes the gradient estimate.

\vspace{2mm}

\noindent {\it Step 1: Construction of the deterministic control.} Let us denote $\mathcal{U}_{\xi}=\mathcal{U}_{\xi}(T,0)$ and set
    \begin{align*}
        &\mathcal{A}_{\xi}\colon L^2(0,T;U)\rightarrow H,\quad \mathcal{A}_{\xi}h=\int_0^T\mathcal{U}_{\xi}(T,t)Bh(t)dt,\\
        &\mathcal{M}_{\xi}\colon H\rightarrow H,\quad \mathcal{M}_{\xi}z=\mathcal{A}_{\xi}\mathcal{A}_{\xi}^*z=\int_0^T\mathcal{U}_{\xi}(T,t)BB^*\mathcal{U}_{\xi}(T,t)^*zdt.
    \end{align*}

    Following the proof of \cite[Theorem 3.1]{LXZ-26}, define two  convex and continuous functions by 
    \begin{align*}
        \Lambda\colon  L^2(0,T;U)\rightarrow\R^+,&\quad \Lambda(h)=\frac{1}{2}\|h\|^2_{L^2(0,T;U)},\\
        \Phi\colon H\rightarrow\R^+,&\quad \Phi(z)=\frac{1}{2\rho}\|z+\mathcal{U}_{\xi}y\|^2.
    \end{align*}

    We then consider the following minimization problem
    \begin{equation}\label{eq problem2}
        \inf_{h\in L^2(0,T;U)}J(h)=\inf_{h\in L^2(0,T;U)}(\Lambda(h)+\Phi(\mathcal{A}_{\xi}h)).
    \end{equation}

    Note that $J$ is strictly convex; thus problem \eqref{eq problem2} admits at most one solution. The Fenchel conjugates $\Lambda^*\colon L^2(0,T;U)\rightarrow\R^+$ and $\Phi^*\colon H\rightarrow\R$ are given by
    \begin{align*}
        \Lambda^*(h)&=\sup_{h'\in L^2(0,T;U)}(\langle h,h'\rangle_{L^2(0,T;U)}-\Lambda(h'))=\frac{1}{2}\|h\|^2_{L^2(0,T;U)}=\Lambda(h),\\
        \Phi^*(z)&=\sup_{z'\in H}(\langle z,z'\rangle-\Phi(z'))=\frac{\rho}{2}\|z\|^2-\langle z,\mathcal{U}_{\xi}y\rangle.
    \end{align*}

    The corresponding dual problem is then given by
    \begin{equation}\label{eq problem2*}
        \sup_{z\in H}J^*(z)=\sup_{z\in H}(-\Phi^*(-z)-\Lambda^*(\mathcal{A}_{\xi}^*z))=-\inf_{z\in H}\hat{J}^*(z),
    \end{equation}
    where 
    \begin{equation*}
        \hat{J}^*\colon H\rightarrow\R,\quad \hat{J}^*(z):=\tfrac{1}{2}\langle(\mathcal{M}_{\xi}+\rho)z,z\rangle+\langle z,\mathcal{U}_{\xi}y\rangle.
    \end{equation*}

   Clearly, $\Phi$ and $\Lambda$ are finite and continuous. Thus, by the Fenchel--Rockafellar duality theorem \cite{ET-99}, both problems \eqref{eq problem2} and \eqref{eq problem2*} admit solutions. In particular, 
    \begin{equation*}
        \inf_{h\in L^2(0,T;U)}J(h)=\sup_{z\in H}J^*(z)=-\inf_{z\in H}\hat{J}^*(z).
    \end{equation*}

    The functional $\hat{J}^*$ is differentiable, and its unique minimizer $z^*\in H$ is given by
    \begin{equation*}
        (\mathcal{M}_{\xi}+\rho)z^*+\mathcal{U}_{\xi}y=0,\quad z^*=-(\mathcal{M}_{\xi}+\rho)^{-1}\mathcal{U}_{\xi}y.
    \end{equation*}
 
    Meanwhile, using the Fenchel dual relation, the unique minimizer $h^*$ of $J$ is given by
    \begin{equation}\label{eq h* def}
            h^*(t)=B^*\mathcal{U}_{\xi}(T,t)^*z^*,\quad t\in(0,T).
    \end{equation}     
    By definition, one has
    \begin{equation*}
        J(h^*)=-\hat{J}^*(z^*)=-\frac{1}{2}\langle z^*,\mathcal{U}_{\xi}y\rangle=-\frac{1}{2}\langle\mathcal{U}_{\xi}^*z^*,y\rangle.
    \end{equation*}
   Also, from our construction, it follows that
    \begin{equation}\label{eq h-beta1}
        \frac{1}{2}\|h^*\|^2_{L^2(0,T;U)}+\frac{1}{2\rho}\|\mathcal{U}_{\xi}y+\mathcal{A}_{\xi}h^*\|^2=J(h^*)=-\frac{1}{2}\langle \mathcal{U}_{\xi}^*z^*,y\rangle.
    \end{equation}   
    Applying the weak observability inequality, i.e. \eqref{eq L obs} yields that
    \begin{equation}\label{eq h-beta2}
    \begin{aligned}
        |\langle \mathcal{U}_{\xi}^*z^*,y\rangle|\leq \|\mathcal{U}_{\xi}^*z^*\| \|y\|\leq \left(C\|B^*\mathcal{U}_{\xi}(T,\cdot)^*z^*\|_{L^2(0,T;U)}+\alpha \|z^*\|\right)\|y\|.
    \end{aligned}          
    \end{equation}    
    In addition, in view of the relation between $h^*$ and $z^*$, one has
    \begin{equation}\label{eq h-beta3}
    \begin{aligned}        \mathcal{U}_{\xi}y+\mathcal{A}_{\xi}h^*&=\mathcal{U}_{\xi}y+\mathcal{A}_{\xi}B^*\mathcal{U}_{\xi}(T,\cdot)^*z^*=-(\mathcal{M}_{\xi}+\rho)z^*+\mathcal{M}_{\xi}z^*=-\rho z^*.
    \end{aligned}
    \end{equation} 

    Summarizing \eqref{eq h* def}-\eqref{eq h-beta3}, we obtain that
    \begin{equation*}
        \|h^*\|^2_{L^2(0,T;U)}+\frac{1}{\rho}\|\mathcal{U}_{\xi}y+\mathcal{A}_{\xi}h^*\|^2\leq \left(C\|h^*\|_{L^2(0,T;U)}+\frac{\alpha}{\rho}\|\mathcal{U}_{\xi}y+\mathcal{A}_{\xi}h^*\|\right)\|y\|.
    \end{equation*}
    Using Young's inequality, one has
     \begin{equation*}
        \|h^*\|^2_{L^2(0,T;U)}+\frac{1}{\rho}\|\mathcal{U}_{\xi}y+\mathcal{A}_{\xi}h^*\|^2\leq \left(C^2+\frac{\alpha^2}{\rho}\right)\|y\|^2,
    \end{equation*}
    which implies that
    \begin{equation*}
        \|\mathcal{U}_{\xi}y+\mathcal{A}_{\xi}h^*\|\leq\left(\alpha+C\sqrt{\rho}\right)\|y\|,\quad\|h^*\|_{L^2(0,T;U)}\leq\left(C+\frac{\alpha}{\sqrt{\rho}}\right)\|y\|.
    \end{equation*}

    Consequently, for any $\alpha'\in(\alpha,1)$, we choose $\rho>0$ such that $\alpha+C\sqrt{\rho}=\alpha'$. This choice is independent of $\xi\in C([0,T];H)$ and $y\in H$. Set $C_0=C+\alpha/\sqrt{\rho}$. In what follows, the control term $h$ in equation \eqref{control equation2} is taken to be $h=h^*$ and has the form
    \begin{equation}\label{eq hz}
        h(t):=h_{\xi}(t)=-B^*\mathcal{U}_{\xi}(T,t)^*(\mathcal{M}_{\xi}+\rho)^{-1}\mathcal{U}_{\xi}y,\quad t\in(0,T),\, y\in H,\; \xi\in C([0,T];H).
    \end{equation}

    \vspace{2mm}
    \noindent {\it Step 2: Application to the stochastic trajectory.} We now consider the linearization of equation \eqref{nl equation}, which is given by
    \begin{equation*}
    \begin{cases}
        dZ=AZdt+\nabla F(X)Zdt,\quad t>0,\\
        Z_0=y.
    \end{cases}
    \end{equation*}

   For $v\in L^2(\Omega;L^2(0,T;U))$, define $V_t:=\langle\mathcal D X_t,v\rangle_{L^2(0,T;U)}$. Then $V$ satisfies
    \begin{equation*}
    \begin{cases}
        dV=AVdt+\nabla F(X)Vdt+Bv(t)dt,\quad t\in(0,T),\\
        V_0=0.
    \end{cases}
    \end{equation*}

   Since $X\in C([0,T];H)$ almost surely, we thus take 
    \begin{equation*}
        v(t)=-h_{X}(t)=B^*\mathcal{U}_X(T,t)^*(\mathcal{M}_X+\rho)^{-1}\mathcal{U}_Xy,\quad t\in(0,T),
    \end{equation*}
    and set $R_t:=Z_t-V_t$. This implies
    \begin{equation*}
        \begin{cases}
        dR=ARdt+\nabla F(X)Rdt+Bh_{X}(t)dt,\quad t\in(0,T),\\
        R_0=y,
    \end{cases}
    \end{equation*}    
    whose mild solution satisfies that
    \begin{equation*}
        R_T=\mathcal{U}_Xy+\mathcal{A}_Xh_{X}.
    \end{equation*}

   We therefore compute  that
   \begin{equation}\label{eq nabla PT}
    \begin{aligned}
        |\langle \nabla P_T\varphi(x),y\rangle|&=\left|\E\langle\nabla\varphi(X_T),Z_T\rangle\right|\leq        |\E\langle\nabla\varphi(X_T),V_T\rangle|+|\E\langle\nabla\varphi(X_T),R_T\rangle|\\
        &=|\E\langle \mathcal{D}(\varphi(X_T)),h_{X}\rangle_{L^2(0,T;U)}|+|\E\langle\nabla\varphi(X_T),R_T\rangle|\\
        &=\left|\E\left(\varphi(X_T)\int_0^Th_{X}(t)dW_t\right)\right|+|\E\langle\nabla\varphi(X_T),R_T\rangle|\\
        &\leq\|\varphi\|_{\infty}\E\left|\int_0^Th_{X}(t)dW_t\right|+\|\nabla\varphi\|_{\infty}\E\|R_T\|.
    \end{aligned}
    \end{equation}     
    Using the construction of the control $h_{X}$, it follows that
    \begin{equation}\label{eq RT}
        \E\|R_T\|\leq \alpha'\|y\|.
    \end{equation}

    \vspace{2mm}
    
    \noindent {\it Step 3: Estimate of the control cost.}  It thus remains to bound the cost of the control. To this end, invoking the generalized Itô isometry, see e.g. \cite{Nualart-06}, one has
    \begin{equation}\label{eq hW}
    \begin{aligned}
        \E\left|\int_0^Th_{X}(t)dW_t\right|^2&=\E\|h_X\|_{L^2(0,T;U)}^2+\E\int_0^T\int_0^T\left\langle\mathcal{D}_sh_{X}(t),\mathcal{D}_th_{X}(s)\right\rangle_{\mathcal{L}_2(U;U)}dtds\\
        &\leq C_0^2\|y\|^2+ \E\int_0^T\int_0^T\|\mathcal{D}_sh_{X}(t)\|_{\mathcal{L}_2(U;U)}^2 dtds,
    \end{aligned}
    \end{equation}
    where $\mathcal{D}_sf$ denotes the evaluation of $\mathcal{D}f$ at time $s\in[0,T]$, i.e. $\mathcal{D}f=\{\mathcal{D}_sf\}_{0\leq s\leq T}$ for any $f\in L^2(\Omega;E)$ on a Hilbert space $E$ with $\mathcal{D}f\in L^2(\Omega;L^2(0,T;\mathcal{L}_2(U;E)))$.

    Additionally, using the expression of $h_X$ by \eqref{eq hz}, we derive 
    \begin{align*}
        \mathcal{D}_sh_{X}(\cdot)&=-\mathcal{D}_s\left(\mathcal{A}_X^*(\mathcal{M}_X+\rho)^{-1}\mathcal{U}_Xy\right)\\
        &=-\left(\mathcal{D}_s\mathcal{A}_X^*\right)(\mathcal{M}_X+\rho)^{-1}\mathcal{U}_Xy-\mathcal{A}_X^*\left(\mathcal{D}_s(\mathcal{M}_X+\rho)^{-1}\right)\mathcal{U}_Xy\\        
        &\quad-\mathcal{A}_X^*(\mathcal{M}_X+\rho)^{-1}\left(\mathcal{D}_s\mathcal{U}_X\right)y.
    \end{align*}

    Note that
    \begin{align*}        \mathcal{D}_s\mathcal{U}_X(T,t)^*&=\left(\mathcal{D}_s\mathcal{U}_X(T,t)\right)^*, \\       \mathcal{D}_s((\mathcal{M}_X+\rho)^{-1})&=-(\mathcal{M}_X+\rho)^{-1}((\mathcal{D}_s\mathcal{A}_X)\mathcal{A}_X^*+\mathcal{A}_X(\mathcal{D}_s\mathcal{A}_X^*))(\mathcal{M}_X+\rho)^{-1}.
    \end{align*}
    Moreover, it follows that
    \begin{align*}
            \|\mathcal{A}^*_X(\mathcal{M}_X+\rho)^{-1/2}\|_{\mathcal{L}(H;L^2(0,T;U))}&\leq1,\\
            \|(\mathcal{M}_X+\rho)^{-1/2}\mathcal{A}_X\|_{\mathcal{L}(L^2(0,T;U);H)}&\leq1,\\
            \|(\mathcal{M}_X+\rho)^{-1/2}\|_{\mathcal{L}(H)}&\leq\rho^{-1/2}.
        \end{align*}
    We also understand $\mathcal{D}_s\mathcal{A}_X^*$ in the following sense:
    for any $q\in H$, $u\in U$ and $v\in L^2(0,T;U)$,
    \begin{equation*}
        \left\langle(\mathcal{D}_s\mathcal{A}_X^*q)u,v\right\rangle_{L^2(0,T;U)}=\left\langle(\mathcal{D}_s\mathcal{A}_Xv)u,q\right\rangle_H.
    \end{equation*}
    With this convention,
    \begin{equation*}
        \|\mathcal{D}_s\mathcal{A}_X^*\|_{\mathcal{L}_2(H;\mathcal{L}_2(U;L^2(0,T;U)))}=\|\mathcal{D}_s\mathcal{A}_X\|_{\mathcal{L}_2(L^2(0,T;U);\mathcal{L}_2(U;H))}.
    \end{equation*}
    In particular, for any $q\in H$,
    \begin{equation*}
        \|\mathcal{D}_s\mathcal{A}_X^*q\|_{\mathcal{L}_2(U;L^2(0,T;U))}\leq\|\mathcal{D}_s\mathcal{A}_X\|_{\mathcal{L}_2(L^2(0,T;U);\mathcal{L}_2(U;H))}\|q\|.
    \end{equation*}

    Summarizing these relations, we derive that
    \begin{align*}
         \|\mathcal{D}_sh_{X}(\cdot)\|_{L^2(0,T;\mathcal{L}_2(U;U))}&=\|\mathcal{D}_sh_{X}(\cdot)\|_{\mathcal{L}_2(U;L^2(0,T;U))}\\
         &\leq \|\left(\mathcal{D}_s\mathcal{A}_X^*\right)(\mathcal{M}_X+\rho)^{-1}\mathcal{U}_Xy\|_{\mathcal{L}_2(U;L^2(0,T;U))}\\
         &\quad+\|\mathcal{A}_X^*\left(\mathcal{D}_s(\mathcal{M}_X+\rho)^{-1}\right)\mathcal{U}_Xy\|_{\mathcal{L}_2(U;L^2(0,T;U))}\\
        &\quad+\|\mathcal{A}_X^*(\mathcal{M}_X+\rho)^{-1}\left(\mathcal{D}_s\mathcal{U}_X\right)y\|_{\mathcal{L}_2(U;L^2(0,T;U))}\\
        &:=I_1+I_2+I_3.    
    \end{align*}

        In particular, one has
    \begin{align*}
        I_1&\leq\|\mathcal{D}_s\mathcal{A}_X^*(\mathcal{M}_X+\rho)^{-1}\mathcal{U}_Xy\|_{\mathcal{L}_2(U;L^2(0,T;U))}
        \leq\|\mathcal{D}_s\mathcal{A}_X\|_{\mathcal{L}_2(L^2(0,T;U);\mathcal{L}_2(U;H))}
        \|(\mathcal{M}_X+\rho)^{-1}\mathcal{U}_Xy\|\\
        &\leq\rho^{-1}\|\mathcal{D}_s\mathcal{A}_X\|_{\mathcal{L}_2(L^2(0,T;U);\mathcal{L}_2(U;H))}
        \|\mathcal{U}_X\|_{\mathcal{L}(H)}\|y\|.
    \end{align*}
    For the second term, 
    \begin{equation*}
        I_2\leq I_{21}+I_{22},
    \end{equation*}
    where
    \begin{align*}
        I_{21}&:=\|\mathcal{A}_X^*(\mathcal{M}_X+\rho)^{-1}(\mathcal{D}_s\mathcal{A}_X)\mathcal{A}_X^*(\mathcal{M}_X+\rho)^{-1}\mathcal{U}_Xy\|_{\mathcal{L}_2(U;L^2(0,T;U))}\\
        &\leq \|\mathcal{A}_X^*(\mathcal{M}_X+\rho)^{-1/2}\|_{\mathcal{L}(H;L^2(0,T;U))}^2\|(\mathcal{M}_X+\rho)^{-1/2}\|_{\mathcal{L}(H)}^2\|\mathcal{D}_s\mathcal{A}_X\|_{\mathcal{L}_2(L^2(0,T;U);\mathcal{L}_2(U;H))}\|\mathcal{U}_Xy\|\\
        &\leq\rho^{-1}\|\mathcal{D}_s\mathcal{A}_X\|_{\mathcal{L}_2(L^2(0,T;U);\mathcal{L}_2(U;H))}\|\mathcal{U}_X\|_{\mathcal{L}(H)}\|y\|,
    \end{align*}
    and 
          \begin{align*}
        I_{22}
        &:=\|\mathcal{A}_X^*(\mathcal{M}_X+\rho)^{-1}\mathcal{A}_X
        (\mathcal{D}_s\mathcal{A}_X^*)(\mathcal{M}_X+\rho)^{-1}\mathcal{U}_Xy
        \|_{\mathcal{L}_2(U;L^2(0,T;U))}\\
        &\leq\|\mathcal{A}_X^*(\mathcal{M}_X+\rho)^{-1/2}\|_{\mathcal{L}(H;L^2(0,T;U))}
        \|(\mathcal{M}_X+\rho)^{-1/2}\mathcal{A}_X\|_{\mathcal{L}(L^2(0,T;U);H)}\\
        &\quad\cdot\|\mathcal{D}_s\mathcal{A}_X^*(\mathcal{M}_X+\rho)^{-1}\mathcal{U}_Xy    \|_{\mathcal{L}_2(U;L^2(0,T;U))}\\
        &\leq\rho^{-1}\|\mathcal{D}_s\mathcal{A}_X\|_{\mathcal{L}_2(L^2(0,T;U);\mathcal{L}_2(U;H))}
        \|\mathcal{U}_X\|_{\mathcal{L}(H)}\|y\|.
    \end{align*}
    Finally, the last term satisfies that
    \begin{align*}
        I_3&\leq \|\mathcal{A}_X^*(\mathcal{M}_X+\rho)^{-1/2}\|_{\mathcal{L}(H;L^2(0,T;U))}\|(\mathcal{M}_X+\rho)^{-1/2}\|_{\mathcal{L}(H)}\|\mathcal{D}_s\mathcal{U}_X\|_{\mathcal{L}(H;\mathcal{L}_2(U;H))}\|y\|\\
        &\leq \rho^{-1/2}\|\mathcal{D}_s\mathcal{U}_X\|_{\mathcal{L}(H;\mathcal{L}_2(U;H))}\|y\|.
    \end{align*}

     Recall that $\mathcal{U}_X=\mathcal{U}_X(T,0)$ is given by \eqref{eq UZ}, and the nonlinear operator $F$ satisfies that $\|\nabla F\|_{\infty}<\infty$. This implies that for any $T>0$, there exists a constant $C_T>0$ such that
    \begin{equation*}
        \|\mathcal{U}_X\|_{\mathcal{L}(H)}\leq C_T\quad \text{ almost surely.}
    \end{equation*}
    
    Collecting these estimates, we obtain that
    \begin{equation}\label{eq DhX}
    \begin{aligned}
        \E\int_0^T\int_0^T\|\mathcal{D}_sh_{X}(t)\|_{\mathcal{L}_2(U;U)}^2 dtds&\leq 18\rho^{-2}C_T^2\|y\|^2\E \int_0^T\|\mathcal{D}_s\mathcal{A}_X\|_{\mathcal{L}_2(L^2(0,T;U);\mathcal{L}_2(U;H))}^2 ds\\
        &\quad +2\rho^{-1}\|y\|^2\E \int_0^T\|\mathcal{D}_s\mathcal{U}_X\|_{\mathcal{L}(H;\mathcal{L}_2(U;H))}^2 ds. 
    \end{aligned}
    \end{equation}
    
    \vspace{2mm}
    \noindent {\it Step 4: Estimates on $\mathcal D_s\mathcal A_X$ and $\mathcal D_s\mathcal U_X$.} For the second term, we compute that for any $0\leq s,t\leq T$ and  $z\in H$,
    \begin{align*}
        (\mathcal{D}_s\mathcal{U}_X(T,t))z&=\int_{s\lor t}^TS(T-r)\left(\nabla F(X_r)\mathcal{D}_s\mathcal{U}_X(r,t)z+\nabla^2 F(X_r)(\mathcal{D}_sX_r,\mathcal{U}_X(r,t)z)\right)dr.
        \end{align*}
    Invoking the assumptions on $F$ and $S$, for $\mathcal{D}_s\mathcal{U}_X$, we derive that
\begin{align*}
    \|(\mathcal{D}_s\mathcal{U}_X(T,t))z\|_{\mathcal{L}_2(U;H)}&\leq\int_{s\lor t}^T\|S(T-r)\nabla F(X_r)\mathcal{D}_s\mathcal{U}_X(r,t)z\|_{\mathcal{L}_2(U;H)}dr\\
    &\quad+\int_{s\lor t}^T\|S(T-r)\nabla^2F(X_r)(\mathcal{D}_sX_r,\mathcal{U}_X(r,t)z)\|_{\mathcal{L}_2(U;H)}dr .
\end{align*}
We first estimate the second term. Since
\begin{equation*}
    \mathcal{U}_X(r,t)z=S(r-t)z+\int_t^rS(r-l)\nabla F(X_l)\mathcal{U}_X(l,t)zdl,\quad r\geq t,
\end{equation*}
the smoothing estimate of $S$ and the boundedness of $\nabla F$ imply
\begin{equation}\label{eq UX V estimate}
    \|\mathcal{U}_X(r,t)z\|_V\leq C_T(r-t)^{-\gamma}\|z\|,\qquad 0\leq t<r\leq T.
\end{equation}
Moreover, for $r<s$, $\mathcal{D}_sX_r=0$; and, for $r\geq s$,
\begin{equation*}
    \mathcal{D}_sX_r=S(r-s)B+\int_s^rS(r-l)\nabla F(X_l)\mathcal{D}_sX_ldl.
\end{equation*}
Again by the smoothing estimate of $S$ and the boundedness of $\nabla F$, we obtain
\begin{equation}\label{eq DsX V estimate}
    \|\mathcal{D}_sX_r\|_{\mathcal L_2(U;V)}\leq C_T(r-s)^{-\gamma},\qquad 0\leq s<r\leq T.
\end{equation}
Indeed, this follows from
\begin{align*}
    \|\mathcal{D}_sX_r\|_{\mathcal L_2(U;V)}&\leq C_T(r-s)^{-\gamma}\|B\|_{\mathcal L_2(U;H)}+C_T\int_s^r(r-l)^{-\gamma}\|\mathcal{D}_sX_l\|_{\mathcal L_2(U;H)}dl,
\end{align*}
together with the standard $H$-estimate for $\mathcal D_sX_l$.

Hence, using \eqref{eq F second V}, \eqref{eq UX V estimate} and \eqref{eq DsX V estimate}, we have
\begin{align*}
    \int_{s\lor t}^T\|S(T-r)\nabla^2F(X_r)(\mathcal{D}_sX_r,\mathcal{U}_X(r,t)z)\|_{\mathcal{L}_2(U;H)}dr&\leq C_T\int_{s\lor t}^T
    \|\mathcal{D}_sX_r\|_{\mathcal L_2(U;V)}
    \|\mathcal{U}_X(r,t)z\|_Vdr\\
    &\leq C_T\|z\|\int_{s\lor t}^T(r-s)^{-\gamma}(r-t)^{-\gamma}dr\\
    &\leq C_T\|z\|,
\end{align*}
where the last inequality follows from $\gamma<1/2$. Therefore,
\begin{align*}
    \|(\mathcal{D}_s\mathcal{U}_X(T,t))z\|_{\mathcal{L}_2(U;H)}&\leq C_T\int_{s\lor t}^T
    \|(\mathcal{D}_s\mathcal{U}_X(r,t))z\|_{\mathcal{L}_2(U;H)}dr+C_T\|z\|.
\end{align*}
An application of Gronwall's lemma then yields that there exists a deterministic constant
$C_T=C(S,F,B,T)$ such that
\begin{equation}\label{eq Ux}
     \|\mathcal{D}_s\mathcal{U}_X(T,t)\|_{\mathcal{L}(H;\mathcal{L}_2(U;H))}\leq C_T,\qquad0\leq s,t\leq T.
\end{equation}
 
    For the term $\mathcal{D}_s\mathcal{A}_X$, one has, for $v\in L^2(0,T;U)$,
\begin{align*}
    (\mathcal{D}_s\mathcal{A}_X)v=\int_s^T\mathcal{U}_X(T,r)\nabla^2 F(X_r)\left(\mathcal{D}_sX_r,\int_0^r\mathcal{U}_X(r,\tau)Bv(\tau)d\tau\right)dr.
\end{align*}
We estimate this term in the integrated Hilbert--Schmidt norm. For $0\leq r\leq T$, define
\begin{equation*}
    \mathcal K_rv:=\int_0^r\mathcal{U}_X(r,\tau)Bv(\tau)d\tau,\qquad v\in L^2(0,T;U).
\end{equation*}
By the smoothing estimate of the linearized evolution and the admissibility of $B$, we have
\begin{equation*}
    \|\mathcal K_r\|_{\mathcal L_2(L^2(0,T;U);V)}^2=\int_0^r\|\mathcal{U}_X(r,\tau)B\|_{\mathcal L_2(U;V)}^2d\tau\leq C_T.
\end{equation*}
Moreover, as above, 
\begin{equation*}
    \|\mathcal D_sX_r\|_{\mathcal L_2(U;V)}\leq C_T(r-s)^{-\gamma},\qquad 0\leq s<r\leq T .
\end{equation*}
Hence, using the assumption \eqref{eq F second V}, the boundedness of $\mathcal U_X(T,r)$ on $H$, and the ideal property of Hilbert--Schmidt operators, we obtain
\begin{align*}
    \|\mathcal{D}_s\mathcal{A}_X\|_{\mathcal L_2(L^2(0,T;U);\mathcal L_2(U;H))}&\leq C_T\int_s^T\|\mathcal D_sX_r\|_{\mathcal L_2(U;V)}\|\mathcal K_r\|_{\mathcal L_2(L^2(0,T;U);V)}dr\leq C_T\int_s^T(r-s)^{-\gamma}dr .
\end{align*}
Therefore, since $\gamma<1/2$, 
\begin{align}\label{eq Ax}
    \mathbb E\int_0^T\|\mathcal{D}_s\mathcal{A}_X\|_{\mathcal L_2(L^2(0,T;U);\mathcal L_2(U;H))}^2ds\leq C_T\int_0^T\left(\int_s^T(r-s)^{-\gamma}dr\right)^2ds\leq C_T .
\end{align}

   \vspace{-1mm}
   
   Consequently, collecting \eqref{eq nabla PT}-\eqref{eq Ax}, we conclude that for any $\alpha'\in(\alpha,1)$, there exists a constant $C'=C'(\alpha',S,F,B,T)$ such that
    \begin{equation*}
        |\nabla P_T\varphi(x)|\leq C'\|\varphi\|_{\infty}+\alpha'\|\nabla\varphi\|_{\infty}.
    \end{equation*}
    for any $x\in H$ and any $\varphi\colon H\rightarrow \R$  with $\|\varphi\|_{\infty}$ and $\|\nabla\varphi\|_{\infty}$ finite. 
    
    This completes the proof of Theorem \ref{thm 3}.

    \end{proof}

\section{The Oseen equation with localized noise}\label{Sec Oseen}

In this section, we consider a stochastic Oseen equation driven by localized white noise. The key point is that the Oseen operator may contain unstable low modes, and hence the asymptotic regularization is not a direct consequence of deterministic exponential stability.

\vspace{2mm}

Let $D\subset\mathbb R^d$ be a bounded domain with smooth boundary, where $d=2$ or $3$. The stochastic Oseen equation with localized white noise is given by
\begin{equation}\label{oseen equation}
\begin{cases}
    \partial_tu-\nu\Delta u+(\overline U\cdot\nabla)u+(u\cdot\nabla)\overline U+\nabla p=\frac{dW}{dt}(t,x),
    \quad x\in D,\; t>0,\\
    \dvg u=0,\quad  u|_{\partial D}=0,\\
    u(0)=u_0.
\end{cases}
\end{equation}
Here $u(t,x)$ denotes the velocity field, $p(t,x)$ denotes the pressure, $\nu>0$ is the viscosity coefficient, and $\overline U=\overline{U}(x)$ is a prescribed stationary background flow satisfying 
\begin{equation*}
    \overline U\in L^\infty(D;\mathbb R^d),\quad \dvg\overline{U}=0\quad\text{in }\mathcal{D}'(D).
\end{equation*}

The localized random force $W(t,x)$ is given by
\begin{equation}\label{eq noise oseen}
    W(t,x)=\sum_{1\leq j\leq N}b_j\beta_j(t)\psi_j(x).
\end{equation} 
Here $N\in\N^+$, $b_j$ are nonzero real numbers, and $\beta_j(t)$ are independent standard Brownian motions. The functions $\psi_j$ are {\it supported in a nonempty smooth subdomain $\omega\subset D$}; see Figure~\ref{fig1}.

Specifically, let $\{\tilde{\psi}_j\}_{j\in\N^+}$ be an orthonormal basis of $L^2(\omega;\mathbb R^d)$ consisting of eigenfunctions of the vector-valued Neumann Laplacian on $\omega$, with eigenvalues $\mu_j$, i.e., 

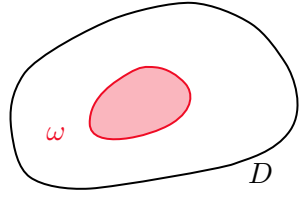
\begin{figure}[htbp]
\noindent
\makebox[\textwidth][l]{
\begin{minipage}[t]{0.48\textwidth}

      \begin{equation*}
     \begin{cases}
         -\Delta_{\omega}\tilde{\psi}_j=\mu_j\tilde{\psi}_j,\quad &x\in \omega,\\
    \partial_\nu\tilde{\psi}_j=0, \quad  &x\in \partial\omega,
     \end{cases}    
    \end{equation*}
    \noindent and 
    \begin{equation*}
        0=\mu_1\leq \mu_2\leq \cdots,\quad \mu_j\rightarrow\infty.
    \end{equation*}
    
\end{minipage}
\hspace{-0.03\textwidth}
\begin{minipage}[t]{0.48\textwidth}  
 \vspace{-1mm}
  \centering
\tikzset{every picture/.style={line width=0.75pt}}      

\begin{tikzpicture}[x=0.75pt,y=0.75pt,yscale=-1,xscale=1] 
 
\draw  [color={rgb, 255:red, 0; green, 0; blue, 0 }  ,draw opacity=1 ][line width=0.75]  (53.76,10.63) .. controls (64.12,6.62) and (83.08,2.06) .. (93.25,2.56) .. controls (103.42,3.07) and (116.07,8.47) .. (123.16,13.15) .. controls (130.25,17.83) and (136.32,26.76) .. (140.35,33.03) .. controls (144.37,39.31) and (150.05,55.37) .. (142.47,66.42) .. controls (134.9,77.48) and (120.58,80.69) .. (114.42,83.1) .. controls (108.25,85.52) and (54.73,94.73) .. (42.43,95.46) .. controls (30.12,96.19) and (18.31,94.08) .. (12.24,89.76) .. controls (6.16,85.44) and (4.1,75.48) .. (2.9,66.07) .. controls (1.71,56.67) and (2.69,47.87) .. (9.08,38.32) .. controls (15.47,28.77) and (43.4,14.64) .. (53.76,10.63) -- cycle ;
\draw [line width=3] [line join = round][line cap = round]    ;
\draw  [color={rgb, 255:red, 239; green, 12; blue, 37 }  ,draw opacity=1 ][fill={rgb, 255:red, 239; green, 12; blue, 37 }  ,fill opacity=0.3 ] (61.54,37.85) .. controls (67.52,34.32) and (69.91,34.32) .. (75.3,34.82) .. controls (80.68,35.33) and (86.18,38.66) .. (90.05,43.03) .. controls (93.91,47.4) and (94.42,55.9) .. (83.27,62.86) .. controls (72.11,69.81) and (42.48,76.49) .. (42.14,62.97) .. controls (41.79,49.44) and (55.55,41.38) .. (61.54,37.85) -- cycle ;

\draw (120,81.4) node [anchor=north west][inner sep=0.75pt]    {$D$};
\draw (19,63.4) node [anchor=north west][inner sep=0.75pt]    {$\textcolor[rgb]{0.94,0.05,0.15}{\omega }$};

\end{tikzpicture}
\vspace{-1mm}
    \caption{Localized noise}\label{fig1}
\end{minipage}}
\end{figure}
\noindent The functions $\psi_j\in L^2(D;\mathbb R^d)$ are then defined as the zero extension of $\tilde{\psi}_j$  by 
\begin{equation*}
    \psi_j(x):=\begin{cases}
        \tilde{\psi}_j(x),\quad &x\in \omega,\\
        0,\quad &x\in D\setminus\omega.
    \end{cases}
\end{equation*}
\begin{remark}
The localized noise directions here are chosen as eigenfunctions of the Neumann Laplacian on the small region $\omega$. This choice is mainly for convenience and is not essential. One may also use other orthonormal bases of $L^2(\omega;\R^d)$, extended by zero outside $\omega$, to define the noise directions. For example, the Dirichlet eigenfunctions are used in {\rm \cite{LXZ-26}}.
\end{remark}

We now specify the functional setting. Let
\begin{equation}\label{eq ossen H}
    \mathcal H\left\{u\in L^2(D;\mathbb R^d):\dvg u=0,\quad u\cdot n|_{\partial D}=0\right\},
\end{equation}
and let $\Pi$ be the Leray projection from $L^2(D;\mathbb R^d)$ onto $ \mathcal H$. Applying $\Pi$ to \eqref{oseen equation}, the pressure term disappears and the equation can be rewritten as
\begin{equation}\label{abstract oseen equation}
\begin{cases}
    du-A_{\overline U}udt=\Pi dW(t),\quad t>0,\\
    u(0)=u_0,
\end{cases}
\end{equation}
where denotes the realization on $\mathcal H$ of the Oseen operator formally given by
\begin{equation*}
    A_{\overline U}u:=\Pi\left(\nu\Delta u-(\overline U\cdot\nabla)u-(u\cdot\nabla)\overline{U}\right),
\end{equation*}
with the first-order terms understood in the distributional sense.

Under these settings, for any $u_0\in\mathcal H$ and $T>0$, equation \eqref{oseen equation} admits a unique mild solution $u(\cdot)\in L^2(\Omega;C([0,T];\mathcal H))$.

\begin{theorem}\label{thm oseen}
For any $T>0$, there exists a constant $C_T>0$ with the following property. For any $\alpha\in(0,1)$, set
\begin{equation}\label{eq N star oseen}
    N_*(T,\alpha):=\min\left\{N\in\N^+:\mu_{N+1}\geq C_T\alpha^{-2}\right\}.
\end{equation}
If the integer $N$ in \eqref{eq noise oseen} satisfies $N\geq N_*$, then the Markov semigroup $\{P_t\}_{t\geq0}$ associated with equation \eqref{oseen equation} on $\mathcal H$ is $\alpha$-asymptotically regular in time $T$.
\end{theorem}

To study the asymptotic strong Feller property of \eqref{oseen equation}, we shall verify the weak observability inequality required in Theorem~\ref{thm 1}. This reduces the problem to an observability estimate for the adjoint Oseen equation
\begin{equation}\label{oseen adjoint}
\begin{cases}
    -\partial_ty-\nu\Delta y-(\overline{U}\cdot\nabla)y-(\nabla y)^{\rm T}\overline{U}+\nabla q=0,
    \quad x\in D,\;t\in(0,T),\\
    \dvg y=0,\quad y|_{\partial D}=0,\\
    y(T)=y_T\in\mathcal H,
\end{cases}
\end{equation}
where $q$ is the adjoint pressure. Recall the following observability inequality for the Oseen system. Indeed, the quantitative estimate in the form of  $C_T=Ce^{C/T}$ is recently derived by Buffe--Takahashi\cite{BT-25}. In \cite[Corollary 1.5]{BT-25}, this result is formulated for
strong solutions; however, by standard arguments, this result can be readily extended to weak solutions.

\begin{lemma}\label{lemma oseen OI}
For any $T>0$,  there exists a constant $C_T>0$
such that, for any $y_T\in\mathcal H$, the solution $y$ of equation  \eqref{oseen adjoint} satisfies
\begin{equation*}
    \|y(0)\|_{\mathcal H}^2\leq C_T\int_0^T\int_{\omega}|y(t,x)|^2dxdt .
\end{equation*}
\end{lemma}

The following proposition is a truncated version of the observability inequality.

\begin{proposition}\label{prop oseen truncated OI}
For any $T>0$, there exists a constant $C_T>0$ such that, for any $N\in\N^+$ and $y_T\in\mathcal H$, the solution $y$ of equation \eqref{oseen adjoint} satisfies
\begin{equation*}
    \|y(0)\|_{\mathcal H}^2\leq C_T\left(\|\mathsf P_N(\mathbf 1_{\omega}y)\|_{L^2(0,T;L^2(\omega))}^2+\mu_{N+1}^{-1}\|y_T\|_{\mathcal H}^2\right),
\end{equation*}
where $\mathsf P_N$ denotes the orthogonal projection from $L^2(\omega;\mathbb R^d)$ onto $\operatorname{span}\{\psi_j:1\leq j\leq N\}$.
\end{proposition}

\begin{proof}
Throughout the proof, the constant $C_T$ is independent of $N$ and $y_T$. By Lemma~\ref{lemma oseen OI}, 
\begin{equation*}
    \|y(0)\|_{\mathcal H}^2\leq C_T\|y\|_{L^2(0,T;L^2(\omega))}^2.
\end{equation*}
Since $\mathsf P_N$ is the orthogonal projection in $L^2(\omega;\mathbb R^d)$, it follows that
\begin{align*}
    \|y(0)\|_{\mathcal H}^2&\leq C_T\left(\|\mathsf P_N(\mathbf 1_{\omega}y)\|_{L^2(0,T;L^2(\omega))}^2+\|(I-\mathsf P_N)(\mathbf 1_{\omega}y)\|_{L^2(0,T;L^2(\omega))}^2
    \right)\\
    &\leq C_T\left(\|\mathsf P_N(\mathbf 1_{\omega}y)\|_{L^2(0,T;L^2(\omega))}^2+\mu_{N+1}^{-1}\|y\|_{L^2(0,T;H^1(D))}^2\right).
\end{align*}
Note that 
\begin{equation*}
    \|(I-\mathsf P_N)(\mathbf 1_{\omega}y)(t)\|_{L^2(\omega)}^2=\sum_{j>N}\langle y(t),\psi_j\rangle_{L^2(\omega)}^2.
\end{equation*}
On the other hand, by the spectral representation of the Neumann Laplacian
on $\omega$,
\begin{equation*}
    \|\nabla y(t)\|_{L^2(\omega)}^2=\sum_{j\in\N^+}\mu_j\langle y(t),\psi_j\rangle_{L^2(\omega)}^2.
\end{equation*}
Hence
\begin{align*}
    \|(I-\mathsf P_N)(\mathbf 1_{\omega}y)(t)\|_{L^2(\omega)}^2\leq\mu_{N+1}^{-1}\|\nabla y(t)\|_{L^2(\omega)}^2\leq\mu_{N+1}^{-1}\|y(t)\|_{H^1(D)}^2.
\end{align*}
Integrating this estimate over $(0,T)$ gives
\begin{equation*}
    \|(I-\mathsf P_N)(\mathbf 1_{\omega}y)\|_{L^2(0,T;L^2(\omega))}^2\leq\mu_{N+1}^{-1}\|y\|_{L^2(0,T;H^1(D))}^2.
\end{equation*}

It remains to estimate the last integral. Let $z(t)=y(T-t)$. Then $z(0)=y_T$ and $z$ solves the forward adjoint Oseen equation
\begin{equation*}
    \partial_tz-\nu\Delta z-(\overline{U}\cdot\nabla)z-(\nabla z)^{\rm T}\overline{U}+\nabla \widetilde q=0,\qquad \dvg z=0,\qquad z|_{\partial D}=0 .
\end{equation*}
Testing this equation by $z$ in $L^2(D;\mathbb R^d)$ yields
\begin{equation*}
    \frac{1}{2}\frac{d}{dt}\|z(t)\|_{\mathcal H}^2+\nu\|\nabla z(t)\|_{L^2(D)}^2=\int_D((\nabla z(t))^{\rm T}\overline{U})\cdot z(t)dx,
\end{equation*}
which implies 
\begin{equation*}
    \frac{d}{dt}\|z(t)\|_{\mathcal H}^2+\nu\|\nabla z(t)\|_{L^2(D)}^2\leq \nu^{-1}\|\overline{U}\|_{L^\infty(D)}^2\|z(t)\|_{\mathcal H}^2.
\end{equation*}
Thus for any $t\geq 0$
\begin{equation*}
    \|z(t)\|_{\mathcal H}^2\leq\exp\left(\nu^{-1}\|\overline{U}\|_{L^\infty(D)}^2t\right)\|y_T\|_{\mathcal H}^2,
\end{equation*}
and
\begin{align*}
    \int_0^T\|\nabla z(t)\|_{L^2(D)}^2dt&\leq\nu^{-1}\left(1+\nu^{-1}\|\overline{U}\|_{L^\infty(D)}^2T \exp\left(\nu^{-1}\|\overline{U}\|_{L^\infty(D)}^2T\right)\right)\|y_T\|_{\mathcal H}^2\leq C_T\|y_T\|_{\mathcal H}^2.
\end{align*}

Combining these estimates, we obtain
\begin{align*}
    \|y(0)\|_{\mathcal H}^2&\leq C_T\left(\|\mathsf P_N(\mathbf 1_{\omega}y)\|_{L^2(0,T;L^2(\omega))}^2+\mu_{N+1}^{-1}\|y_T\|_{\mathcal H}^2\right).
\end{align*}
\end{proof}

Invoking Proposition \ref{prop oseen truncated OI} and Theorem \ref{thm 1}, we now prove Theorem~\ref{thm oseen}.

\begin{proof}[Proof of Theorem~\ref{thm oseen}]
To align with the abstract setting in Section~\ref{Sec 1}, recall that the state space $\mathcal H$ is given by \eqref{eq ossen H}. By \eqref{abstract oseen equation}, the equation can be written as
\begin{equation*}
    du=Audt+BdW_t^N,\quad u(0)=u_0\in \mathcal H,
\end{equation*}
where
\begin{align*}
    A=A_{\overline{U}}\text{ with the no-slip boundary condition},&\quad U=\mathbb R^N,\quad
    W_t^N=(\beta_1(t),\cdots,\beta_N(t)),\\
    B\colon\mathbb R^N\rightarrow \mathcal H,\quad B\xi=\sum_{1\leq j\leq N}b_j\xi_j\Pi\psi_j,&\quad\xi=(\xi_1,\cdots,\xi_N)\in\mathbb R^N .
\end{align*}
Then $B\in\mathcal L(\mathbb R^N;\mathcal H)$. Since $U$ is finite dimensional, the noise is $\mathcal H$-valued and the Hilbert--Schmidt admissibility condition in \eqref{abs equation} is automatically satisfied. Thus equation \eqref{oseen equation} fits into the abstract setting of \eqref{abs equation}. For any $z\in\mathcal H$, by the definition of $B$, one has
\begin{equation}\label{eq B star oseen}
    B^*z=\left(b_1\langle z,\Pi\psi_1\rangle_{\mathcal H},\cdots,b_N\langle z,\Pi\psi_N\rangle_{\mathcal H}\right),\quad \|B^*z\|_{\mathbb R^N}^2=\sum_{1\leq j\leq N}b_j^2\langle z,\psi_j\rangle_{L^2(\omega)}^2.
\end{equation}

Let $y$ be the solution of equation \eqref{oseen adjoint}. Applying Proposition~\ref{prop oseen truncated OI} with $N\geq N_*(T,\alpha)$, one has
\begin{align*}
    \|y(0)\|_{\mathcal H}^2&\leq C_T\left(\|\mathsf P_N(\mathbf 1_{\omega}y)\|_{L^2(0,T;L^2(\omega))}^2+\mu_{N+1}^{-1}\|y_T\|_{\mathcal H}^2\right)\\
    &\leq C_T\|\mathsf P_N(\mathbf 1_{\omega}y)\|_{L^2(0,T;L^2(\omega))}^2+\alpha^2\|y_T\|_{\mathcal H}^2,
\end{align*}
where the last inequality follows from \eqref{eq N star oseen}.

Set $\hat b_N:=\min_{1\leq j\leq N}|b_j|$. By the assumption on $\{b_j\}_{1\leq j\leq N}$, we have $\hat b_N>0$. Since $\mathsf P_N$ is the orthogonal projection from $L^2(\omega;\mathbb R^d)$ onto $\operatorname{span}\{\psi_j:1\leq j\leq N\}$, one has
\begin{align*}
    \|\mathsf P_N(\mathbf 1_{\omega}y)\|_{L^2(0,T;L^2(\omega))}^2&=\int_0^T\sum_{1\leq j\leq N}\langle y(t),\psi_j\rangle_{L^2(\omega)}^2dt\leq\hat b_N^{-2}\int_0^T\sum_{1\leq j\leq N}
    b_j^2\langle y(t),\psi_j\rangle_{L^2(\omega)}^2dt.
\end{align*}
Using \eqref{eq B star oseen} with $z=y(t)$, we infer that
\begin{align*}
    \|y(0)\|_{\mathcal H}^2&\leq C_T\hat b_N^{-2}\int_0^T\|B^*y(t)\|_{\mathbb R^N}^2dt+\alpha^2\|y_T\|_{\mathcal H}^2=C_T\hat b_N^{-2}\|B^*y\|_{L^2(0,T;\mathbb R^N)}^2+\alpha^2\|y_T\|_{\mathcal H}^2.
\end{align*}
Since $y(t)=S(T-t)^*y_T$, this is exactly the weak observability inequality required in Theorem~\ref{thm 1}. This completes the proof.
\end{proof}

\section{Non-autonomous linear equation with localized noise}\label{Sec 5}
 
In this section, we apply the non-autonomous linear criterion from Section~\ref{Sec 2} to uniformly parabolic equations with localized finite-dimensional noise, reducing asymptotic regularization to a quantitative finite-mode observability estimate.

\vspace{2mm}

Let us consider a non-autonomous linear parabolic equation driven by localized white noise, 
\begin{equation}\label{para equation}
\begin{cases}
    \partial_tu-\mathcal L(t)u=\frac{dW}{dt}(t,x),\quad x\in D,\; t>s\geq0,\\
    u|_{\partial D}=0,\\
    u(s)=u_s.
\end{cases}
\end{equation}
Here $D\subset\mathbb R^d$ is a bounded domain with smooth boundary, $d\in\N^+$, $u_s\in L^2(D)$, and $\mathcal L(t)$ is a non-autonomous uniformly parabolic operator of the form
\begin{equation*}
    \mathcal L(t)u=\sum_{i,j=1}^d \partial_{x_i}\left(a_{ij}(t,x)\partial_{x_j}u\right)+\sum_{j=1}^d a_j(t,x)\partial_{x_j}u+a_0(t,x)u.
\end{equation*}
We assume that the coefficients  satisfy that, for any $T>0$,
\begin{equation*}
    a_{ij}\in C^{1,2}([0,T]\times \overline{D}),\quad a_{ij}=a_{ji},\quad a_j\in C^{0,1}([0,T]\times \overline{D}),\quad a_0\in L^\infty((0,T)\times D).
\end{equation*}
Moreover, the principal part is uniformly elliptic: there exists $\theta>0$ such that
\begin{equation*}
    \sum_{i,j=1}^d a_{ij}(t,x)\xi_i\xi_j\geq\theta|\xi|^2,\qquad(t,x,\xi)\in [0,\infty)\times D\times\mathbb R^d .
\end{equation*}

The localized random force $W(t,x)$ is given by
\begin{equation}\label{eq noise para}
    W(t,x)=\sum_{1\leq j\leq N}b_j\beta_j(t)\psi_j(x).
\end{equation}
Here $N\in\N^+$, $b_j$ are nonzero real numbers, and $\beta_j(t)$ are independent standard Brownian motions. The functions $\psi_j$ are supported in a nonempty smooth subdomain $\omega\subset D$. More precisely, as in Section~\ref{Sec Oseen}, let $\{\tilde{\psi}_j\}_{j\in\N^+}$ be an orthonormal basis of $L^2(\omega)$ consisting of eigenfunctions of the Neumann Laplacian on $\omega$, with eigenvalues $\mu_j$, $\mu_j\rightarrow\infty$. Then set each $\psi_j\in L^2(D)$ to be the zero extension of $\tilde{\psi}_j$.

Under these settings, for any $u_s\in L^2(D)$ and $T>s\geq0$, equation \eqref{para equation} admits a unique mild solution $u(\cdot)\in L^2(\Omega;C([s,T];L^2(D)))$.

\begin{theorem}\label{thm 5}
For any $T>s\geq 0$, there exists a constant $C_{s,T}>0$ with the following property. For any $\alpha\in(0,1)$, set
\begin{equation}\label{eq N star para}
    N_*(s,T,\alpha):=\min\left\{N\in\N^+:\mu_{N+1}\geq C_{s,T}\alpha^{-2}\right\}.
\end{equation}
If the integer $N$ in \eqref{eq noise para} satisfies $N\geq N_*$, then the Markov family $\{P_{s,t}\}_{t\geq s\geq0}$ associated with equation \eqref{para equation} on $L^2(D)$ is $\alpha$-asymptotically regular from time $s$ to time $T$.
\end{theorem}

As in the previous section, the proof relies on a weak observability estimate for the adjoint equation. More precisely, we consider
\begin{equation}\label{para adjoint}
\begin{cases}
    -\partial_ty-\mathcal L(t)^*y=0,\quad x\in D,\;t\in(s,T),\\
    y|_{\partial D}=0,\\
    y(T)=y_T\in L^2(D).
\end{cases}
\end{equation}
Here
\begin{equation*}
    \mathcal L(t)^*y=\sum_{i,j=1}^d \partial_{x_j}\left(a_{ij}(t,x)\partial_{x_i}y\right)-\sum_{j=1}^d\partial_{x_j}\big(a_j(t,x)y\big)+a_0(t,x)y.
\end{equation*}

We first recall the following observability estimate obtained from the parabolic Carleman inequality in \cite[Lemma 1.2]{FI-96}.

\begin{lemma}\label{lemma para OI}
For any $T>s\geq0$, there exists a constant $C_{s,T}>0$ such that, for any
$y_T\in L^2(D)$, the solution $y$ of equation \eqref{para adjoint} satisfies
\begin{equation*}
    \|y(s)\|_{L^2(D)}^2\leq C_{s,T}\int_s^{T}\int_{\omega}|y(t,x)|^2dxdt.
\end{equation*}
\end{lemma}

We next derive a truncated version of the observability inequality.

\begin{proposition}\label{prop para truncated OI}
For any $T>s\geq0$, there exists a constant $C_{s,T}>0$ such that, for any
$N\in\N^+$ and $y_T\in L^2(D)$, the solution $y$ of equation \eqref{para adjoint} satisfies
\begin{equation*}
    \|y(s)\|_{L^2(D)}^2\leq C_{s,T}\left(\|\mathsf P_N(\mathbf 1_{\omega}y)\|_{L^2(s,T;L^2(\omega))}^2+\mu_{N+1}^{-1}\|y_T\|_{L^2(D)}^2\right),
\end{equation*}
where $\mathsf P_N$ denotes the orthogonal projection from $L^2(\omega)$ onto
$\operatorname{span}\{\psi_j:1\leq j\leq N\}$.
\end{proposition}

\begin{proof}
By Lemma \ref{lemma para OI}, we have
\begin{equation*}
    \|y(s)\|_{L^2(D)}^2\leq C_{s,T}\|y\|_{L^2(s,T;L^2(\omega))}^2.
\end{equation*}
Since $\mathsf P_N$ is the orthogonal projection in $L^2(\omega)$, it follows that
\begin{align*}
    \|y(s)\|_{L^2(D)}^2&\leq C_{s,T}\left(\|\mathsf P_N(\mathbf 1_{\omega}y)\|_{L^2(s,T;L^2(\omega))}^2+\|(I-\mathsf P_N)(\mathbf 1_{\omega}y)\|_{L^2(s,T;L^2(\omega))}^2\right)\\
    &\leq C_{s,T}\left(\|\mathsf P_N(\mathbf 1_{\omega}y)\|_{L^2(s,T;L^2(\omega))}^2+\mu_{N+1}^{-1}\|y\|_{L^2(s,T;H^1(D))}^2\right).
\end{align*}
Here the estimate
\begin{equation*}
    \|(I-\mathsf P_N)(\mathbf 1_{\omega}y)\|_{L^2(s,T;L^2(\omega))}^2\leq\mu_{N+1}^{-1}\|y\|_{L^2(s,T;H^1(D))}^2 
\end{equation*}
follows from the same arguments as in the proof of Proposition \ref{prop oseen truncated OI}. Moreover, by the standard energy estimate for equation \eqref{para adjoint}, one has
\begin{equation*}
    \|y\|_{L^2(s,T;H^1(D))}^2\leq C_{s,T}\|y_T\|_{L^2(D)}^2.
\end{equation*}
Indeed, after the time reversal $z(r)=y(T-r)$, the function $z$ solves a forward uniformly parabolic equation with initial data $z(0)=y_T$. Testing it by $z$ and using the ellipticity of the principal part, and the boundedness of the lower order coefficients, we derive
\begin{equation*}
    \sup_{0\leq r\leq T-s}\|z(r)\|_{L^2(D)}^2+\int_0^{T-s}\|z(r)\|_{H^1(D)}^2dr\leq C_{s,T}\|y_T\|_{L^2(D)}^2.
\end{equation*}
Returning to $y$ and combining the above estimates then completes the proof.
\end{proof}

Combining the truncated observability inequality with the non-autonomous linear criterion obtained in Section~\ref{Sec 2}, we now prove Theorem~\ref{thm 5}.

\begin{proof}[Proof of Theorem~\ref{thm 5}]
To align with the abstract setting in Section~\ref{Sec 2}, set
\begin{equation*}
    H=L^2(D),\quad U=\mathbb R^N,\quad G(t)=0,\quad W_t^N=(\beta_1(t),\cdots,\beta_N(t)).
\end{equation*}
Let $A(t)=\mathcal L(t)$ with the Dirichlet boundary condition on $D$. Under the assumptions on the coefficients of $\mathcal L(t)$, the family $\{A(t)\}_{t\geq0}$ generates an evolution family $\{U(t,s):0\leq s\leq t\}$ on $L^2(D)$. Moreover, equation \eqref{para equation} can be written as
\begin{equation*}
    dX(t)=A(t)X(t)dt+BdW_t^N,\quad X(s)=u_s,
\end{equation*}
where
\begin{equation*}
    B\colon\mathbb R^N\rightarrow L^2(D),\quad B\xi=\sum_{1\leq j\leq N}b_j\xi_j\psi_j, \qquad\xi=(\xi_1,\cdots,\xi_N)\in\mathbb R^N .
\end{equation*}
Then $B\in\mathcal L(\mathbb R^N;L^2(D))$. Since $U$ is finite dimensional, the noise is $L^2(D)$-valued and the Hilbert--Schmidt admissibility condition in \eqref{abs equation2} is automatically satisfied. Thus equation \eqref{para equation} fits into the abstract setting of \eqref{abs equation2}.

\vspace{2mm}
Applying Proposition~\ref{prop para truncated OI} with $N=N_*$ and using $N\geq N_*$, we obtain
\begin{align*}
    \|y(s)\|_{L^2(D)}^2&\leq C_{s,T}\left(\|\mathsf P_{N_*}(\mathbf 1_{\omega}y)\|_{L^2(s,T;L^2(\omega))}^2+\mu_{N_*+1}^{-1}\|y_T\|_{L^2(D)}^2\right) \\
    &\leq C_{s,T}\|\mathsf P_N(\mathbf 1_{\omega}y)\|_{L^2(s,T;L^2(\omega))}^2+\alpha^2\|y_T\|_{L^2(D)}^2,
\end{align*}
where the last inequality follows from \eqref{eq N star para}.

Set $\hat b_N:=\min_{1\leq j\leq N}|b_j|$. By the assumption on $\{b_j\}_{1\leq j\leq N}$, $\hat b_N>0$. Moreover,
\begin{align*}
    \|\mathsf P_N(\mathbf 1_{\omega}y)\|_{L^2(s,T;L^2(\omega))}^2=\int_s^T\sum_{1\leq j\leq N}\langle y(t),\psi_j\rangle_{L^2(\omega)}^2dt\leq\hat b_N^{-2}\int_s^T\sum_{1\leq j\leq N}b_j^2\langle y(t),\psi_j\rangle_{L^2(\omega)}^2dt.
\end{align*}
For any $z\in L^2(D)$, by the definition of $B$, one has
\begin{equation*}
    B^*z=\left(b_1\langle z,\psi_1\rangle_{L^2(D)},\cdots,b_N\langle z,\psi_N\rangle_{L^2(D)}\right).
\end{equation*}
Therefore,
\begin{equation*}
    \|y(s)\|_{L^2(D)}^2\leq C_{s,T}\hat b_N^{-2}\|B^*y\|_{L^2(s,T;\mathbb R^N)}^2+\alpha^2\|y_T\|_{L^2(D)}^2.
\end{equation*}
This implies the weak observability inequality required in Theorem~\ref{thm 2}. The conclusion then follows from Theorem~\ref{thm 2}. The proof is complete.
\end{proof}

\section{The parabolic Sine--Gordon equation with localized noise} \label{Sec 6}

In this section, we consider a stochastic parabolic Sine--Gordon equation and show that the abstract criterion in Theorem~\ref{thm 3} can be applied to deduce the asymptotic strong Feller property.

\vspace{2mm}

The parabolic Sine--Gordon equation with localized white noise  is given by
\begin{equation}\label{sine equation}
\begin{cases}
   \partial_{t}u-\Delta u+\kappa\sin u=\frac{dW}{dt}(t,x),\quad x\in D,\; t>0,\\
    u|_{\partial D}=0,\\
    u(0)=u_0.
\end{cases}
\end{equation}
Here $D\subset\mathbb R^3$ is a bounded domain with smooth boundary, $u_0\in L^2(D)$ and $\kappa\in\R$ is an arbitrary parameter.

The localized random force $W(t,x)$ is given by
\begin{equation}\label{eq noise sine}
    W(t,x)=\sum_{1\leq j\leq N}b_j\beta_j(t)\psi_j(x).
\end{equation}
Here $N\in\N^+$, $b_j$ are nonzero real numbers, and $\beta_j(t)$ are independent standard Brownian motions. The functions $\psi_j$ are supported in a nonempty smooth subdomain $\omega\subset D$. More precisely, as in Section~\ref{Sec Oseen}, let $\{\tilde{\psi}_j\}_{j\in\N^+}$ be an orthonormal basis of $L^2(\omega)$ consisting of eigenfunctions of the Neumann Laplacian on $\omega$, with eigenvalues $\mu_j$, $\mu_j\rightarrow\infty$. Then set each $\psi_j\in L^2(D)$ to be the zero extension of $\tilde{\psi}_j$.

Under these settings,   for any $u_0\in L^2(D)$ and $T>0$, equation \eqref{sine equation} admits a unique mild solution $u(\cdot)\in L^2(\Omega;C([0,T];L^2(D)))$.

\begin{theorem}\label{thm 4}
There exists a constant $C>0$ with the following property. For any $T>0$, $\kappa\in\R$ and $\alpha\in(0,1)$, set
\begin{equation*}
    N_*(\kappa,T,\alpha):=\min\left\{N\in\N^+:\mu_{N+1}\geq C\alpha^{-2}\exp\left(C\left(T^{-1}+T|\kappa|+|\kappa|^{2/3}\right)\right)\right\}.
\end{equation*}
If the integer $N$ in \eqref{eq noise sine} satisfies $N\geq N_*$, then the Markov semigroup $\{P_t\}_{t\geq0}$ associated with equation \eqref{sine equation} on $L^2(D)$ is $\alpha'$-asymptotically regular in time $T$ for any $\alpha'\in(\alpha,1)$.
\end{theorem}

As in the previous sections, we consider the adjoint equation of the linearized parabolic Sine--Gordon equation. Let $g\in L^\infty((0,T)\times D)$ and consider
\begin{equation}\label{y equation}
\begin{cases}
    -\partial_ty-\Delta y+gy=0,\quad x\in D,\;t\in(0,T),\\
    y|_{\partial D}=0,\\
    y(T)=y_T\in L^2(D).
\end{cases}
\end{equation}

We recall the following observability inequality, established in  \cite[Theorem 1.2]{FCZ-00}.

\begin{lemma}\label{lemma OI}
There exists a constant $C>0$, such that for any $T>0$, 
$g\in L^\infty((0,T)\times D)$ and $y_T\in L^2(D)$, the solution $y$ of equation \eqref{y equation} satisfies
\begin{equation*}
    \|y(0)\|_{L^2(D)}^2\leq C\exp\left(C\left(T^{-1}+T\|g\|_{\infty}+\|g\|_{\infty}^{2/3}\right)\right)  \int_0^T\int_{\omega}|y(t,x)|^2dxdt.
\end{equation*}
\end{lemma}

The following proposition is a truncated version of the observability inequality.

\begin{proposition}\label{prop truncated OI}
There exists a constant $C>0$, such that for any
$T>0$, $g\in L^\infty((0,T)\times D)$, $N\in\N^+$ and $y_T\in L^2(D)$, the solution $y$ of equation \eqref{y equation} satisfies
\begin{equation*}
    \|y(0)\|_{L^2(D)}^2\leq C\exp\left(C\left(T^{-1}+T\|g\|_{\infty}+\|g\|_{\infty}^{2/3}\right)\right)\left(\|\mathsf P_N(\mathbf 1_{\omega}y)\|_{L^2(0,T;L^2(\omega))}^2+\mu_{N+1}^{-1}\|y_T\|_{L^2(D)}^2\right),
\end{equation*}
where $\mathsf P_N$ denotes the orthogonal projection from $L^2(\omega)$ onto
$\operatorname{span}\{\psi_j:1\leq j\leq N\}$.
\end{proposition}
 
\begin{proof}
Set $K_{T,g}:=C\exp\left(C\left(T^{-1}+T\|g\|_{\infty}+\|g\|_{\infty}^{2/3}\right)\right)$. By Lemma \ref{lemma OI}, we have
\begin{align*}
    \|y(0)\|_{L^2(D)}^2&\leq K_{T,g}\|y\|_{L^2(0,T;L^2(\omega))}^2\\
    &=K_{T,g}\left(\|\mathsf P_N(\mathbf 1_{\omega}y)\|_{L^2(0,T;L^2(\omega))}^2+\|(I-\mathsf P_N)(\mathbf 1_{\omega}y)\|_{L^2(0,T;L^2(\omega))}^2\right)\\
    &\leq K_{T,g}\left(\|\mathsf P_N(\mathbf 1_{\omega}y)\|_{L^2(0,T;L^2(\omega))}^2+\mu_{N+1}^{-1}\|y\|_{L^2(0,T;H^1(D))}^2\right).
\end{align*}
Here the estimate
\begin{equation*}
    \|(I-\mathsf P_N)(\mathbf 1_{\omega}y)\|_{L^2(0,T;L^2(\omega))}^2\leq\mu_{N+1}^{-1}\|y\|_{L^2(0,T;H^1(D))}^2 
\end{equation*}
follows from the same arguments as in the proof of Proposition \ref{prop oseen truncated OI}. Moreover, by the standard energy estimate for the backward parabolic equation \eqref{y equation}, one has
\begin{equation*}
    \|y\|_{L^2(0,T;H^1(D))}^2\leq C\exp(CT\|g\|_{\infty})\|y_T\|_{L^2(D)}^2.
\end{equation*}
Indeed, this follows by reversing time and testing the resulting forward parabolic equation by the solution itself. Combining the above estimates and enlarging the constant $C$ in $K_{T,g}$ if necessary, we obtain the desired inequality. The proof is complete.
\end{proof}

Combining the truncated observability inequality with the abstract criterion obtained in Section~\ref{Sec 3}, we now prove Theorem~\ref{thm 4}.

\begin{proof}[Proof of Theorem~\ref{thm 4}]
We verify the assumptions of Theorem~\ref{thm 3}. Set $H=L^2(D)$, $U=\mathbb R^N$, and let $A=\Delta$ with the Dirichlet boundary condition. Choose $V=D((-A)^{\rho/2})$ with $3/4<\rho<1$. Then $V\hookrightarrow L^4(D)$ and
$\|S(t)\|_{\mathcal L(H;V)}\leq C_Tt^{-\rho/2}$ for $0<t\leq T$; hence the smoothing assumption holds with $\gamma=\rho/2<1/2$.

For the parabolic Sine--Gordon equation, $F(u)=-\kappa\sin u$. Thus $F\colon H\to H$ is globally Lipschitz and has at most linear growth. Moreover,
$\nabla F(u)h=-\kappa\cos u\,h$ and
$\nabla^2F(u)(h_1,h_2)=\kappa\sin u\,h_1h_2$. Since $|\sin u|\leq1$ and $V\hookrightarrow L^4(D)$,
\begin{equation*}
    \|\nabla^2F(u)(h_1,h_2)\|_{L^2(D)}\leq C|\kappa|\|h_1\|_V\|h_2\|_V,\qquad h_1,h_2\in V .
\end{equation*}

Writing $W_t^N=(\beta_1(t),\cdots,\beta_N(t))$, define
\begin{equation*}
    B\xi=\sum_{1\leq j\leq N}b_j\xi_j\psi_j,\qquad\xi=(\xi_1,\cdots,\xi_N)\in\mathbb R^N.
\end{equation*}
Since $U=\mathbb R^N$ is finite dimensional, $B\in\mathcal L_2(U;H)$. Hence equation \eqref{sine equation} fits into the abstract framework of \eqref{nl equation}.

It remains to verify the weak observability condition. Let $\xi\in C([0,T];L^2(D))$ be arbitrary. The linearized equation along $\xi$ is
\begin{equation*}
    \partial_tv-\Delta v+\kappa\cos\xi\,v=0,\qquad v|_{\partial D}=0 .
\end{equation*}
The corresponding adjoint equation is \eqref{y equation} with $g=\kappa\cos\xi$, and therefore $ \|g\|_{L^\infty((0,T)\times D)}\leq|\kappa|$. 

By Proposition~\ref{prop truncated OI}, the definition of $B^*$, the fact that $b_j\neq0$ for $1\leq j\leq N$, and the choice of $N_*(\kappa,T,\alpha)$, the same argument as in the proofs of Theorems~\ref{thm oseen} and~\ref{thm 5} yields the $\alpha$-weak observability inequality \eqref{eq L obs}, uniformly in $\xi$. Thus all assumptions of Theorem~\ref{thm 3} are satisfied, and the Markov semigroup associated with \eqref{sine equation} is $\alpha'$-asymptotically regular in time $T$ for every $\alpha'\in(\alpha,1)$. This completes the proof.
\end{proof}
    
    \normalem
    \bibliographystyle{plain}
    \bibliography{References}

\end{document}